\pdfoutput=1
\documentclass[11pt]{amsart}
\usepackage{amssymb, amsmath, amsthm}
\usepackage{color}
\usepackage{graphicx}
\usepackage[final]{hyperref}
\usepackage[a4paper, centering]{geometry}
\newtheorem{teor}{Theorem}[section]
\newtheorem{prop}[teor]{Proposition}
\newtheorem{lemma}[teor]{Lemma}
\newtheorem{cor}[teor]{Corollary}
\theoremstyle{definition}
\newtheorem{dhef}[teor]{Definition}
\theoremstyle{remark}
\newtheorem{rk}[teor]{Remark}
\newtheorem{ehse}[teor]{Example}

\long\def\elimina#1{} 

\def\Leb{\mathcal{L}}
\def\R{\mathbb{R}}

\def\N{\mathbb{N}}

\def\pscal#1#2{\left\langle#1,\,#2\right\rangle}

\def\projm{\Pi^M}
\def\gauge{\rho}

\def\pgauge{\gauge^0}
\def\len{l}
\def\vf{v_f}

\def\uo{u_{\phi}}

\def\uf{u_{f}}
\def\gaph{\Gamma_{\phi}}
\def\gaf{\Gamma_{f}}
\def\gafc{\overline{\Gamma}_{f}}
\def\gaphc{\overline{\Gamma}_{\phi}}
\def\test{C^{\infty}_c\left(\R^n \setminus \gafc\right)}

\def\pscal#1#2{\langle#1,\,#2\rangle}
\def\oray#1{{]\!]#1[\![}}
\def\cray#1{{[\![#1]\!]}}
\def\ccurve#1{{\Gamma}_{#1}}
\def\len#1{L(#1)}
\def\spaceX{X}
\def\Xf{X_f}

\DeclareMathOperator{\inte}{int}
 \DeclareMathOperator{\dive}{div}
 
\DeclareMathOperator{\proj}{\Pi} 
\DeclareMathOperator{\spt}{spt}
\DeclareMathOperator{\essspt}{spt}
\DeclareMathOperator{\diam}{diam}

\begin{document}

%%%%%%%%%%%%%%%%%%%%%%%%%%%%%%%%%%%%%%%%%%%%%%%%%%%%%%%%%%
%amsart format
%\title[Sandpiles in a flat table with a vertical rim]
%{Equilibrium configurations for sandpiles \\ in a flat table with a %vertical rim}%
\title[A BVP in granular matter theory]%
{Existence and uniqueness of solutions \\ for a boundary value problem
\\ arising from granular matter theory}%
\author[G.~Crasta]{Graziano Crasta}
\address{Dipartimento di Matematica ``G.\ Castelnuovo'', Univ.\ di Roma I\\
P.le A.\ Moro 2 -- 00185 Roma (Italy)}
\email[Graziano Crasta]{crasta@mat.uniroma1.it}
\author[A.~Malusa]{Annalisa Malusa}
\email[Annalisa Malusa]{malusa@mat.uniroma1.it}
\keywords{Boundary value problems,
mass transfer theory}
\subjclass[2010]{Primary 35A02, Secondary 35J25}
\date{October 26, 2012}

\begin{abstract}
We consider a system of PDEs of Monge-Kantorovich type that,
in the isotropic case, describes the stationary configurations
of two-layers models in granular matter theory with a general source
and a general boundary data.
We propose a new weak formulation which is consistent with the
physical model and permits us to prove existence and uniqueness
results.
\end{abstract}

% amsart format
\maketitle

%%%%%%%%%%%%%%%%%%%%%%%%%%%%%%%%%%%%%%%%%%%%%%
\section{Introduction}

The model system usually considered for the description of the stationary
configurations of sandpiles on a container is the Monge-Kantorovich type
system of PDEs
\begin{equation}\label{f:MKintro}
\begin{cases}
-\dive(v\, Du) = f & \text{in $\Omega$},\\
|Du|\leq 1,\ v\geq 0 & \text{in $\Omega$},\\
(1-|Du|)v=0 & \text{in $\Omega$},\\
u\leq\phi & \text{on $\partial\Omega$},\\
u=\phi & \text{on $\gaf$}
%\min\limits_{\partial\Omega} (\phi-u) = 0
\end{cases}
\end{equation}
(see, e.g., \cite{AEW,CaCa,CCS,HK}).
The data of the problem are the flat surface of the container $\Omega\subseteq \R^2$, the profile
of the rim $\phi$, and the density of the source $f\geq 0$,
whereas the set $\gaf$ is a subset of $\partial\Omega$, defined in terms
of the other data, that will be
specified below.

The dynamical behaviour of the granular matter is pictured by the pair
$(u,v)$, where $u$ is the profile of the standing layer, whose slope has not to exceed a critical value
($|Du|\leq 1$)
in order to prevent avalanches, while $v\geq 0$ is the thickness of the rolling layer.
The condition $(1-|Du|)v=0$ corresponds to require that
the matter runs down only in the region where the slope of the heaps is maximal.

The set $\gaf$ (which depends on the source $f$, the geometry of $\Omega$, and on the boundary datum
$\phi$) is the part of the border where every admissible profile $u$ touches the rim, in such a way the exceeding sand can fall down (see Definitions in Section \ref{s:problem}).
We underline that the set $\gaf$ is not an additional datum of the problem, but it is constructed in terms of the other data
(see \eqref{f:gafg} for its precise definition).

The main contribution of our results to the theory concerns the uniqueness of the $v$--component for general boundary value problems, based on a new weak formulation of
the continuity equation $-\dive (v\, Du)=f$ in $\Omega$.

The case of the open table problem, corresponding to $u=\phi=0$ on $\partial \Omega$, is already
completely understood (see e.g.\ \cite{CaCa,CCCG,CCS,Pr} and the references therein). Namely, if $d_{\Omega}$ denotes the distance function from the boundary of $\Omega$, it is possible to construct a function $\vf\geq 0$, $\vf\in L^1(\Omega)$ such that the pair $(d_{\Omega},\vf)$
is a solution to \eqref{f:MKintro} (we underline that the continuity equation is  understood in the sense of distributions). Moreover it turns out that $\vf$ is the unique  admissible $v$--component, and every profile $u$ must coincides with 
$d_{\Omega}$ where the transport is active.

These results validate the model for the open table problem, since they depicted the sole 
physically acceptable situation: the mass transport density $v$ has to be uniquely determined  by the data of the problem, while the profile $u$ could be different from the maximal one only where the mass
transportation does not act.

Moreover the profile is unique (and maximal) if and only if the source $f$  pours sand along the ridge of the maximal profile (i.e.\ on the closure of the set where $d_{\Omega}$ is
not differentiable).

As far as we know, only the following two particular cases of non-homogeneous boundary conditions were considered
in literature.

In \cite{CFV}, mostly devoted to a numerical point of view, the problem of the open table with walls (corresponding to $u=\phi=0$ on a regular portion $\Gamma$ of $\partial \Omega$, and 
$\phi=+\infty$ in $\partial \Omega \setminus \Gamma$) is considered. In order to take into account the fact that the sand can flow out from the table only through $\Gamma$, the 
weak formulation of the continuity equation proposed in \cite{CFV} is the following:
\[
\int_\Omega v \pscal{Du}{D\psi}\, dx = \int_\Omega f \psi \, dx\,,
\qquad \forall
\psi\in C^{\infty}_c(\R^2 \setminus \overline{\Gamma}).
\]
Under suitable regularity assumptions on the geometry of the sandpile, it is proved that there exists a function
$\vf\geq 0$, $\vf\in L^1(\Omega)$ such that the pair $(d_{\Gamma},\vf)$ is a solution to \eqref{f:MKintro},
where $d_{\Gamma}$ is the distance function from $\Gamma$.

A different approach to non-homogeneous  boundary conditions was recently proposed in \cite{CMl}. 
In that paper we considered only admissible boundary data, that is continuous functions $\phi$ on $\partial \Omega$ that coincide on the boundary with the related 
Lax--Hopf function $\uo$. In the model,  this corresponds to treating the so called tray table problem, where the boundary datum $\phi$ gives the height of the rim.
The requirements are that the border of the rim is always reached ($u=\phi$ on $\partial\Omega$), and that the continuity equation is satisfied in the sense of distributions.
The existence of a solution is obtained, in analogy with the open table problem, by exhibiting an
explicit function $\vf\geq 0$, $\vf\in L^1(\Omega)$ such that the pair $(\uo,\vf)$ is a solution to \eqref{f:MKintro}. Moreover a necessary and sufficient condition for the uniqueness of the
$u$-component can be obtained, with minor changes, as in the homogeneous case.

The main novelty in the analysis of the non-homogeneous case concerns the lack of uniqueness
of the $v$--component. Namely, the boundary datum $\phi$  modifies the geometry of the
directions along which the sand falls down. 
In particular it may happen that 
a family of transport rays passing across $\Omega$
covers a set of positive measure,
so that it is possible to transport any additional mass
along these rays, keeping the total flux unchanged.

In the present paper we shall deal with general boundary data, thus allowing the presence of walls
as well as of exit points at different heights. The main goal will be to modify the 
%rigorous meaning 
weak formulation
of the continuity equation in order to gain the uniqueness of the $v$--component, without loosing information concerning the $u$--component, then validating the model in a very
general case.

Moreover we shall not require that the profiles have to reach the height of the rim at every point of 
$\partial \Omega$ where they a--priori could agree, compatibly with their gradient constraint 
(i.e.\ at every point where the maximal profile $\uo$ agrees with $\phi$). It is perfectly clear that,
during the evolution ending with the stationary state, the sandpile grows under the action of the source, so that it is not reasonable to require that $u=\phi$ in the part of the boundary not 
reached by those transport rays along which no sand is poured.
For this reason we relax the boundary condition, requiring $u\leq \phi$ on $\partial\Omega$,
and by selecting the region $\gaf\subseteq \partial \Omega$ where $u=\phi$ in terms of $f$
and of the geometry of the transport rays.

The region $\gaf$ also dictates the test functions in the weak formulation of the continuity
equation. Namely, we require that a solution $(u,v)$ to \eqref{f:MKintro} has to satisfy
\[
\int_\Omega v \pscal{Du}{D\psi}\, dx = \int_\Omega f \psi \, dx\,,
\qquad \forall
\psi\in\test, 
\]
thus taking into account the fact that the sand cannot exit from 
$\partial\Omega \setminus \gaf$.
%We remark that this new formulation allows us to recover the %existence and uniqueness results valid for the open table problem.

\medskip
We present all the results in a more general setting, which takes into account the possibility
of homogeneous anisotropies
(see also \cite{CMf,CMg,CMi,CMh}).
More precisely, we shall consider the following system of PDEs
in the open,
bounded and connected set $\Omega\subseteq \R^n$
with Lipschitz boundary:
\begin{equation}\label{f:MKint}
\begin{cases}
-\dive(v\, D\gauge(Du)) = f
&\text{in $\Omega$},\\
\gauge(Du)\leq 1,\ v\geq 0
&\text{in $\Omega$},\\
(1-\gauge(Du))v=0 &\text{in $\Omega$},\\
u\leq\phi & \text{on $\partial\Omega$},\\
u=\phi & \text{on $\gaf$},\\
%\min\limits_{\partial\Omega} (\phi-u) = 0,
\end{cases}
\end{equation}
in the unknowns $v\in L^1(\Omega)$,
$u\in W^{1,\infty}(\Omega)$.
(Here and in the following we understand that
$u\in W^{1,\infty}(\Omega)$ denotes the Lipschitz extension
to $\overline{\Omega}$ of $u$.)
%$u\in W^{1,\infty}(\Omega)\cap C(\overline{\Omega})$.
In this formulation:
\begin{itemize}
\item[-] $\gauge \colon \R^n \to [0,+\infty)$ is the gauge function of
a compact convex set $K\subseteq \R^n$, of class $C^1$ and
containing the origin in its interior;
%$K\subseteq \R^n$, a compact convex set of class $C^1$
%containing the origin in its interior, and its (strictly convex) polar set $K^0$.
%The set $K$ appears in (\ref{f:MKint}) by means of its gauge function $\gauge \colon \R^n \to [0,+\infty)$.
%Moreover the closure of %(connected components of)
%$\Omega$ will be equipped with a (possibly asymmetric) geodesic distance $d_L$ given in terms of the gauge function $\pgauge$ of the polar set $K^0$;
\item[-] $f\in L^1(\Omega)$, $f\geq 0$; 
%the set of non-negative integrable functions in $\Omega$;
\item[-] $\phi\colon \partial\Omega \to \R^+$ is a lower semicontinuous function, $\phi\not\equiv +\infty$.
\end{itemize}

The plan of the paper is the following.
After recalling some notation an basic results,
in Section~\ref{s:problem} we give the precise formulation of
the problem, 
showing that under the assumptions listed above
we do not have, in general, neither existence nor uniqueness of
solutions (see Examples~\ref{r:es1} and~\ref{r:es2}).
In order to overcome these obstructions we then introduce an
additional geometric assumption (see (H5) below), 
which guarantees that the mass is transported straight to the
boundary.
This condition is automatically satisfied if $\phi = 0$,
while in the general case may fail,
possibly causing both concentration of mass transportation
on sets of lower dimension (described by measure-type transport densities)
or branching of transport paths (and thus multiplicity of
transport densities).
%which is related to similar
%assumptions on the structure of the transport rays made in
%\cite{CFV} and \cite{CMl}.
In Section~\ref{s:exi} we prove that \eqref{f:MKint} always
admit solutions, showing that there exists a non-negative function
$v_f\in L^1(\Omega)$ such that the pair $(\uo, v_f)$ is a solution
(being $\uo$ the Lax-Hopf function defined in \eqref{f:LH} below). 
The main ingredient for this step is a disintegration formula
for the Lebesgue measure proved by S.~Bianchini in \cite{Bi}.
In addition, we prove a preliminary but fundamental uniqueness result:
if $(\uo, v)$ is a solution to \eqref{f:MKint},
then $v = v_f$.
In order to get this result,
we show that we need to strengthen 
the geometric assumption (H5) 
(see Example~\ref{r:es3} and assumption (H6) below).

Section~\ref{s:uni} is devoted to the characterization of all
solutions and, consequently, to the uniqueness result.
More precisely, we show that there exists a minimal profile
$u_f$ such that every solution to \eqref{f:MKint} is of the form
$(u, v_f)$ with $u_f\leq u\leq \uo$. %and $u$ $1$-Lipschitz.
In particular, the $v$-component is unique, whereas the $u$-component
is unique (and coincides with $\uo$) if and only if the support
of the source $f$ covers the set of the endpoints of the transport
rays.
Finally, in Section~\ref{s:sand} we briefly rewrite our results in  the isotropic case \eqref{f:MKintro}, and we discuss their interpretation
in terms of granular matter models.

%%%%%%%%%%%%%%%%%%%%%%%%%%%%%%%%%%%%%%%%%%%%%%%%

\section{Notation and preliminaries}
\label{s:prel}

\textsl{General notation.}
The standard scalar product of $x,y\in\R^n$
will be denoted by $\pscal{x}{y}$, while $|x|$ will denote the
Euclidean norm of $x$.
Concerning the segment joining $x$ with $y$, we set
\[
\cray{x,y} := \{tx+(1-t)y;\ t\in [0,1]\},
\qquad \oray{x,y} := \cray{x,y}\setminus\{x,y\}.
%:= \{tx+(1-t)y;\ t\in (0,1)\}.
\]

Given a set $A\subset \R^n$, its interior, its closure and its boundary
will be denoted by $\inte A$, $\overline{A}$ and $\partial A$ respectively.

We shall denote by $\Leb^n$ and $\mathcal{H}^k$ respectively
the $n$-dimensional Lebesgue measure and the
$k$-dimensional Hausdorff measure.
Given a measure $\mu$ and a $\mu$-measurable set $F$,
the symbol $\mu\lfloor F$ will denote the restriction of $\mu$
to the set $F$.

If $g\colon\Omega\to\R$ is a measurable function,
we shall denote by $\essspt g$ the essential support of $g$,
that is the complement in $\Omega$ of the union of all relatively
open subsets $A\subset\Omega$ such that $g=0$ a.e.\ in $A$.
Notice that $\essspt g$ is a relatively closed set in $\Omega$,
but need not to be closed as a subset of $\R^n$.

\smallskip
\textsl{Convex geometry.}
Let us now fix the notation and the basic results concerning
the convex set which plays the r\"ole of gradient constraint
for the $u$-component in (\ref{f:MKint}).
In the following we shall assume that
\begin{equation}\label{f:hypk}
\text{$K$ is a compact, convex subset of $\R^n$ of class $C^1$, with $0\in\inte K$.}
\end{equation}
Let us denote by $K^0$ the
polar set of $K$, that is
\[
K^0 := \{p\in\R^n;\ \pscal{p}{x}\leq 1\ \forall x\in K\}\,.
\]
We recall that, if $K$ satisfies (\ref{f:hypk}), then
$K^0$ is a compact, strictly convex subset of $\R^n$
containing the origin in its interior, and $K^{00} = (K^0)^0 = K$
(see, e.g., \cite{Sch}).

The gauge function $\gauge\colon \R^n \to \R$ of $K$ is defined by
\[
\gauge(\xi) := \inf\{ t\geq 0;\ \xi\in t K\}=\max \{\pscal{\xi}{\eta},\ \eta\in K^0\}\,,
\quad \xi\in\R^n\,.
\]
It is straightforward to see that $\gauge$ is a positively 1-homogeneous convex function such that
$K=\{\xi\in\R^n\colon\ \gauge(\xi)\leq 1\}$. The gauge function of the set $K^0$ will be denoted by $\pgauge$.

The properties of the gauge functions needed in the paper are collected in the following
theorem.

\begin{teor}\label{t:sch}
Assume that $K\subseteq \R^n$ satisfies (\ref{f:hypk}).
Then the following hold:

\par\noindent (i) $\gauge$ is continuously
differentiable in $\R^n\setminus \{0\}$, and
\[
\gauge(\xi+\eta)\leq \gauge(\xi) + \gauge(\eta)\, \quad \forall\ \xi,\eta\in\R^n\,.
\]

\par\noindent (ii) $K^0$ is strictly convex,  and
\[
\begin{split}
\pgauge(\xi+\eta) & \leq \pgauge(\xi) + \pgauge(\eta),  \quad \forall\ \xi,\eta\in\R^n \\
\pgauge(\xi+\eta) &= \pgauge(\xi) + \pgauge(\eta)\ \Leftrightarrow\
\exists \lambda\geq 0\ \colon\ \xi = \lambda\, \eta\ \text{or}\ \eta = \lambda\, \xi.
\end{split}
\]

\par\noindent (iii)
For every $\xi\neq 0$,
$D\gauge(\xi)$ belongs to $\partial K^0$, and
\[
\pscal{D\gauge(\xi)}{\xi} = \gauge(\xi),\quad
\pscal{p}{\xi} < \gauge(\xi)\ \forall p\in K^0,\ p\neq D\gauge(\xi)\,.
\]
%\par\noindent (iii)
%$\pgauge$ is differentiable at $\xi\neq 0$ if and only
%if there exists a unique $\eta \in K$ such that $\pscal{\eta}{\xi}=\pgauge(\xi)$,
%and in this case $\eta= D\pgauge(\xi)$.
\end{teor}

\begin{proof}
See \cite{Sch}, Section 1.7.
\end{proof}

In what follows we shall consider $\R^n$ endowed with the possibly asymmetric norm $\pgauge(x-y)$,
$x,y\in\R^n$.
By Theorem~\ref{t:sch}(ii), the unit ball $K^0$ of $\pgauge$ is strictly convex but,
under the sole assumption (\ref{f:hypk}), it need not be differentiable.
Moreover, the Minkowski structure $(\R^n, \pgauge)$
is not a metric space in the usual sense, since $\pgauge$
need not be symmetric (for an introduction to non-symmetric metrics see \cite{Gromov}).
Finally, since
$K^0$ is compact and $0\in \inte K^0$, then the convex metric is equivalent to the Euclidean one,
that is there exist $c_1$, $c_2>0$ such that $c_1|\xi|\leq \pgauge(\xi)\leq c_2 |\xi|$ for every $\xi\in\R^n$.

\smallskip
\textsl{Curves.}
In the following $\Omega$ will denote an open, bounded, connected subset 
of $\R^n$ with Lipschitz boundary.
Let us denote by $\ccurve{y,x}$ the family of absolutely continuous paths in $\overline{\Omega}$
connecting $y$ to $x$:
\[
\ccurve{y,x}:=\{\gamma\in AC([0,1],\overline{\Omega}),\ \gamma(0)=y,\ \gamma(1)=x\}\,.
\]
For every absolutely continuous curve $\gamma\colon [0,1]\to \R^n$, let us denote by
$\len{\gamma}$ its length with respect to the convex metric associated to $\pgauge$, that is
\[
\len{\gamma} := \int_0^1 \pgauge(\gamma'(t))\, dt\,.
\]

Since $\overline{\Omega}$ is a compact subset of $\R^n$,
by a standard compactness argument we have that
for every $x,y\in\overline{\Omega}$
%if $\ccurve{y,x}\neq\emptyset$, then
there exists a
(distance) minimizing curve $\tilde{\gamma}\in \ccurve{y,x}$
such that
$\len{\tilde{\gamma}}\leq \len{\gamma}$ for every $\gamma\in \ccurve{y,x}$
(see e.g.\ \cite[Thm.~4.3.2]{AmTi}, \cite[\S14.1]{Ces}).

The main motivation for introducing the convex metric associated to $\pgauge$ is the fact that the
Sobolev functions with the gradient constrained to belong to $K$ are
the locally 1-Lipschitz functions with respect to $\pgauge$, as stated in the following result (see \cite[Chap.~5]{Li}).

\begin{lemma}\label{l:diseqrho}
Assume that the set $K\subset\R^n$ satisfies \eqref{f:hypk}.
Let $\pgauge$ be the gauge function of $K^0$, let $\Omega\subset\R^n$ be a Lipschitz domain,
and let $u\colon \Omega \to \R$.
Then the following properties are equivalent.
\begin{itemize}
\item[(i)] $u$  is a locally $1$-Lipschitz function with respect to $\pgauge$, i.e.
\begin{equation}\label{f:lipur}
u(x_2)-u(x_1)\leq\pgauge(x_2-x_1)\quad \text{for every}\
\cray{x_1,x_2}\subset\Omega.
\end{equation}
\item[(ii)]
$u\in W^{1,\infty}(\Omega)$, and $Du(x) \in K$ for a.e.\ $x\in\Omega$.
\item[(iii)]
$u(x) - u(y) \leq \len{\gamma}$
for every $x,y\in\Omega$ and every $\gamma\in\ccurve{y,x}$.
\end{itemize}
\end{lemma}

%%%%%%%%%%%%%%%%%%%%%%%%%%%%%%%%%%%%%%%%%%%%%%%%%%%%%%%%%%%

\section{Formulation of the problem}\label{s:problem}

In this section we shall give the definition of
solution $(u,v)$ to the PDEs system
\begin{equation}
\label{f:MK}
\begin{cases}
-\dive(v\, D\gauge(Du)) = f
&\text{in $\Omega$},\\
\gauge(Du)\leq 1,\ v\geq 0
&\text{in $\Omega$},\\
(1-\gauge(Du))v=0 &\text{in $\Omega$},\\
u\leq\phi & \text{on $\partial\Omega$},\\
u=\phi & \text{on $\gaf$.}
%\min\limits_{\partial\Omega} (\phi-u) = 0\,.
\end{cases}
\end{equation}
%in the open, bounded and connected set $\Omega\subseteq \R^n$.
The basic assumptions are:
\begin{itemize}
\item[(H1)] $\Omega$ is an open, bounded, connected subset of
$\R^n$ with Lipschitz boundary;
\item[(H2)] $\gauge$ is the gauge function of a convex set $K\subseteq \R^n$ satisfying (\ref{f:hypk});
\item[(H3)] $f$ belongs to $L^1_+(\Omega)$, the set of non-negative integrable functions in $\Omega$;
\item[(H4)] $\phi\colon \partial\Omega \to (-\infty,+\infty]$ is a 
lower semicontinuous (l.s.c.)
function, $\phi\not\equiv +\infty$.
\end{itemize}
As we shall see in Examples~\ref{r:es1}, \ref{r:es2} and
\ref{r:es3}, this set of assumptions is not enough in order
to have existence and uniqueness of solutions.
Two additional assumptions on the geometry of the problem
will be introduced in the remaining part of this section.

The set $\gaf\subseteq\partial\Omega$ where the $u$ component is forced to agree with the boundary datum $\phi$,
is dictated by the data of the problem. 
Using the terminology
of the Optimal Transport Theory, $\gaf$ corresponds to the set of initial points of those transport rays on which the transport is active. 
For a rigorous definition, some additional notation is in order.

Let $\uo\colon\overline{\Omega}\to\R$ be
the Lax-Hopf function defined by
\begin{equation}\label{f:LH}
\uo(x) := \inf\left\{\phi(y)+L(\gamma):\ y\in\partial\Omega,\ \gamma\in\ccurve{y,x}
\right\},\quad
x\in \overline{\Omega}.
\end{equation}
%Notice that the function $\uo$ depends only on the geometry of $\Omega$
%and on the boundary datum.
It is clear that for every $x\in \Omega$ the infimum in (\ref{f:LH})
is attained, that is there exist $y\in\partial\Omega$ and a minimizing curve
$\gamma\in\ccurve{y,x}$ such that $\uo(x)=\phi(y)+L(\gamma)$.

%The curve $\gamma$
%corresponds, in the sandpile model, to the transport ray through $x$, that
%is the path that the sand poured in $x$ follows for being carried away from $\Omega$.

\begin{dhef}
\label{d:maxg}
For $x\in\overline{\Omega}$ we call
\textsl{geodesic through} $x$ any
%%%nontrivial ???
curve $\gamma\in\ccurve{y,x}$,
$y\in\partial\Omega$,
satisfying $\uo(x) = \phi(y) + L(\gamma)$.
Moreover, we say that a geodesic through $x$ is (forward) \textsl{maximal}
if its image is not a proper subset of the image
of another geodesic through $x$.
\end{dhef}

We recall that $\uo$ is a Lipschitz function in $\overline{\Omega}$,
$\gauge(D\uo) = 1$ a.e.\ in $\Omega$,
and it is the maximal function in the space $\spaceX$ defined by
%\begin{equation}\label{f:X}
\[
\spaceX := \left\{
u\in W^{1,\infty}(\Omega):\
Du\in K \text{ a.e.~in } \Omega,\
u\leq \phi\ \text{on}\ \partial\Omega
\right\}\,.
\]
%\end{equation}
(Since $\Omega$ has a Lipschitz boundary, we understand that
all functions in $W^{1,\infty}(\Omega)$ are extended to
Lipschitz continuous functions in $\overline{\Omega}$.)
We shall show that the Lax-Hopf function $\uo$ is always an
admissible $u$-component in \eqref{f:MK}, regardless of the source $f$ (see Theorem \ref{t:eu1}).
%(By admissible profile we mean the $u$-component of a solution
%$(u,v)$ to \eqref{f:MK}.)

The function $\uo$ dictates the geometry of the transportation,
that is the mass produced by the source $f$ runs along
the geodesics associated to $\uo$, and
falls down at their initial points.

\smallskip

%%%%%%%%%%%%%%%%%%%%%%%%%%%%%%%%%%%%%%%%%%%%%%%%%%%%%%%%%%%%%%%%%
For every $x\in\overline{\Omega}$ we denote by $\Pi(x)$ the full set of projections
of $x$ on $\partial\Omega$, i.e.
\begin{equation}
\label{f:proj}
\Pi(x) := \left\{y\in\partial\Omega:\
\exists\ \text{a geodesic}\  \gamma\in\ccurve{y,x}
\ \text{ through $x$}
%\ \text{s.t.}\ \uo(x) = \phi(y) + L(\gamma)
\right\}\,,
\end{equation}
whereas $\projm(x)$ will denote the set of maximal projections, i.e.
%\begin{equation}\label{f:projm}
\[
\projm(x) := \left\{y\in\partial\Omega:\
\exists\ \text{a maximal geodesic}\ \gamma\in\ccurve{y,x}
\ \text{ through $x$}\right\}.
\]
%\end{equation}

Let us define the set of the initial points of maximal geodesics
\begin{equation}\label{f:gaphg}
\gaph :=\{y\in \partial\Omega\colon\ \exists \
\text{a maximal geodesic}\ \gamma\in\ccurve{y,x},\
x\in \Omega
%x\neq y,\
%\text{s.t.}\ \uo(x)=\phi(y)+\pgauge(x-y)
\}
= \bigcup_{x\in \Omega} \projm(x)
\end{equation}
and the subset of $\gaph$
of the initial points of those maximal geodesics where
the source $f$ is active 
\begin{equation}\label{f:gafg}
\gaf := %\text{cl}
\{y\in \partial\Omega\colon\ \exists \
\text{maximal\ geodesic}\ \gamma\in\ccurve{y,x},\
x\in\essspt f
\}= \bigcup_{x\in\essspt f} \projm(x)\,.
\end{equation}
The set $\gaf$ turns out to be, in the sandpile problem,
the actual portion of $\gaph$ (depending on the
source of matter $f$) where the sand falls down.
The set $\partial\Omega\setminus \gaph$ is the part of the boundary
closed by the walls.

We are now in a position to fix the rigorous meaning of problem (\ref{f:MK}).
The functional setting for the unknowns $(u,v)$ in (\ref{f:MK}) is
$\Xf \times L^1_+(\Omega)$, where
%\begin{equation}\label{f:Xf}
\[
\Xf := \left\{
u\in\spaceX:\
%u\in W^{1,\infty}(\Omega):\
%Du\in K \text{ a.e.~in } \Omega,\ u\leq\phi
u = \phi\ \text{on}\ \gaf
\right\}\,, \quad
L^1_+(\Omega) := \{v\in L^1(\Omega):\ v\geq 0 \text{ a.e.\ in } \Omega\}\,.
\]
%\end{equation}

\begin{dhef}\label{d:sol}
A pair $(u,v)$ is a solution to \eqref{f:MK} if
\begin{itemize}
\item[(i)] $(u,v)\in \Xf \times L^1_+(\Omega)$;
\item[(ii)] $(1-\gauge(Du))v=0$ a.e.\ in $\Omega$;
\item[(iii)] for every $\psi\in \test$
\[
\int_\Omega v \pscal{D\gauge(Du)}{D\psi}\, dx = \int_\Omega f \psi \, dx\,.
\]
\end{itemize}
\end{dhef}

\begin{rk}
This notion of solution generalizes the one given in \cite{CFV} for the table problem with walls.
The weak formulation in (iii) corresponds to the continuity
equation for the optimal transport problem subjected to
the condition that the mass can flow away from $\Omega$
only through $\gaf$.
\end{rk}

%%%%%%%%%%%%%%%%%%%%%%%%%%%%%%%%%%%%%%%%%%%%%%%%%%%%%%%%%%%%%%%%

The following two examples illustrate some points that should
be taken into account in order to deal with existence and
uniqueness results for \eqref{f:MK}. In order not to interrupt the main flow
of the exposition, the details are postponed to Section~\ref{s:examples}.

\begin{figure}
\centering
\def\svgwidth{10cm}
%\input{nonuniq.pdf_tex}
%% Creator: Inkscape 0.48+devel, www.inkscape.org
%% Accompanies image file 'nonuniq.pdf' (pdf, eps, ps)
\begingroup%
  \makeatletter%
  \providecommand\color[2][]{%
    \errmessage{(Inkscape) Color is used for the text in Inkscape, but the package 'color.sty' is not loaded}%
    \renewcommand\color[2][]{}%
  }%
  \providecommand\transparent[1]{%
    \errmessage{(Inkscape) Transparency is used (non-zero) for the text in Inkscape, but the package 'transparent.sty' is not loaded}%
    \renewcommand\transparent[1]{}%
  }%
  \providecommand\rotatebox[2]{#2}%
  \ifx\svgwidth\undefined%
    \setlength{\unitlength}{302.95687522bp}%
    \ifx\svgscale\undefined%
      \relax%
    \else%
      \setlength{\unitlength}{\unitlength * \real{\svgscale}}%
    \fi%
  \else%
    \setlength{\unitlength}{\svgwidth}%
  \fi%
  \global\let\svgwidth\undefined%
  \global\let\svgscale\undefined%
  \makeatother%
  \begin{picture}(1,0.45203226)%
    \put(0,0){\includegraphics[width=\unitlength]{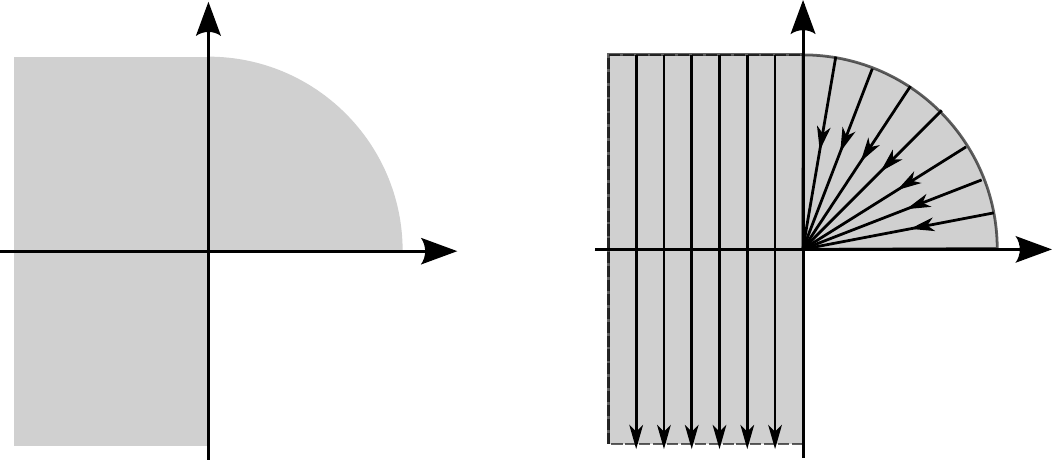}}%
    \put(0.25237365,0.29038069){\color[rgb]{0,0,0}\makebox(0,0)[lb]{\smash{$\Omega_1$}}}%
    \put(0.05770995,0.22742333){\color[rgb]{0,0,0}\makebox(0,0)[lb]{\smash{$\Omega_2$}}}%
    \put(0.08658494,0.00465722){\color[rgb]{0,0,0}\makebox(0,0)[lb]{\smash{$S$}}}%
  \end{picture}%
\endgroup%
\caption{The set $\Omega$ in Example~\ref{r:es1}}
\label{fig:nonuniq}
\end{figure}
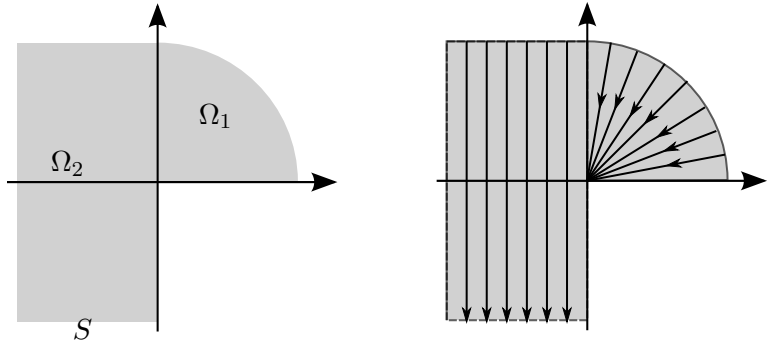

\begin{ehse}
\label{r:es1}
%comment: the mass arriving at $0$ from $\Omega_1$ cannot be
%distributed to $\gaph$
Under our general assumptions (H1)--(H4), problem \eqref{f:MK}
need not admit a solution in the sense of Definition~\ref{d:sol}.
Namely, let $\Omega := \Omega_1\cup\Omega_2\subset\R^2$, where
\[
\begin{split}
\Omega_1 & := \left\{
(r\cos\theta,\, r\sin\theta):\
0 < \theta \leq \frac{\pi}{2},\
0 < r < 1\right\}\,,\\
\Omega_2 & := \left\{x=(x_1,x_2)\in\R^2:\
-1 < x_1 < 0,\
|x_2| < 1\right\}
\end{split}
\]
(see Figure~\ref{fig:nonuniq}).
Let $\phi\colon\partial\Omega\to\R$ be the function
\[
\phi(x)=
\begin{cases}
0 , & x\in S:=[-1,0]\times\{-1\}, \\
+\infty, & \text{otherwise}.
\end{cases}
\]
Let us consider the isotropic case ($\rho(\xi) = |\xi|$)
with a constant source $f\equiv 1$.
The Lax-Hopf function is
\[
\uo(x) =
\begin{cases}
1 + |x|, &\text{if}\ x\in\Omega_1,\\
1 + x_2, &\text{if}\ x\in\Omega_2,
\end{cases}
\]
so that $\gaph=\gaf = (-1,0]\times\{-1\}$, and the geometry of the geodesics is the one depicted in Figure~\ref{fig:nonuniq} right.
In particular  the mass collected in the region
$\Omega_1$ is transported to the origin, and from the origin
there is a unique transport ray $R$
going to $\gaph$.
The concentration of the mass on the single ray $R$
corresponds to the fact that the transport density $v$
should be a measure with a singular part concentrated
on $R$, 
%The transport of that mass along a single ray could be done
%only if $v$ is allowed to be a measure
%; more precisely,
%what is needed is a transport density
%\[
%v = v_{ac}\, \Leb^2 + \frac{\pi}{4}\, \mathcal{H}^1\lfloor R
%\]
which is not allowed in Definition~\ref{d:sol}.
\end{ehse}

\begin{ehse}
\label{r:es2}
Even in the case of existence of solutions to \eqref{f:MK},
the uniqueness of the transport density $v$
may fail if geodesics can bifurcate in the interior.
Namely, let $\Omega := \Omega_1\cup\Omega_2\cup\Omega_3\subset\R^2$, where
\[
\begin{split}
\Omega_1 & := \left\{x=(x_1,x_2)\in\R^2:\
-1 < x_1 < 0,\
|x_2| < 1\right\},\\
\Omega_2 & := \left\{
(r\cos\theta,\, r\sin\theta):\
\frac{\pi}{4} < \theta \leq \frac{\pi}{2},\
0 < r < 1\right\}\,,\\
\Omega_3 & := \left\{
(r\cos\theta,\, r\sin\theta):\
-\frac{\pi}{2} \leq \theta < -\frac{\pi}{4},\
0 < r < 1\right\}
\end{split}
\]
(see Figure~\ref{fig:nonuniq2}).
\begin{figure}
\centering
\def\svgwidth{10cm}
%\input{nonuniq2.pdf_tex}
%% Creator: Inkscape 0.48+devel, www.inkscape.org
%% Accompanies image file 'nonuniq2.pdf' (pdf, eps, ps)
\begingroup%
  \makeatletter%
  \providecommand\color[2][]{%
    \errmessage{(Inkscape) Color is used for the text in Inkscape, but the package 'color.sty' is not loaded}%
    \renewcommand\color[2][]{}%
  }%
  \providecommand\transparent[1]{%
    \errmessage{(Inkscape) Transparency is used (non-zero) for the text in Inkscape, but the package 'transparent.sty' is not loaded}%
    \renewcommand\transparent[1]{}%
  }%
  \providecommand\rotatebox[2]{#2}%
  \ifx\svgwidth\undefined%
    \setlength{\unitlength}{271.17345485bp}%
    \ifx\svgscale\undefined%
      \relax%
    \else%
      \setlength{\unitlength}{\unitlength * \real{\svgscale}}%
    \fi%
  \else%
    \setlength{\unitlength}{\svgwidth}%
  \fi%
  \global\let\svgwidth\undefined%
  \global\let\svgscale\undefined%
  \makeatother%
  \begin{picture}(1,0.52389013)%
    \put(0,0){\includegraphics[width=\unitlength]{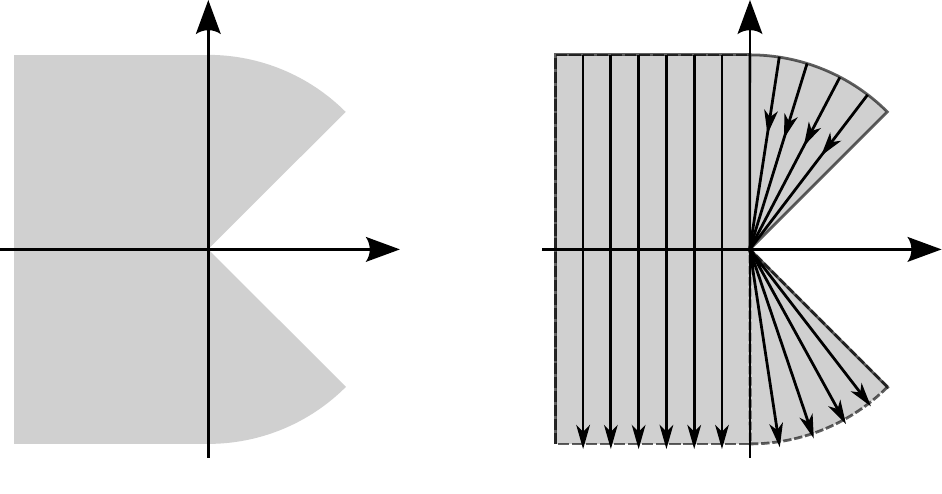}}%
    \put(0.07350109,0.27386184){\color[rgb]{0,0,0}\makebox(0,0)[lb]{\smash{$\Omega_1$}}}%
    \put(0.0882518,0.00834913){\color[rgb]{0,0,0}\makebox(0,0)[lb]{\smash{$S_1$}}}%
    \put(0.3095124,0.05260125){\color[rgb]{0,0,0}\makebox(0,0)[lb]{\smash{$S_3$}}}%
    \put(0.24638486,0.37539449){\color[rgb]{0,0,0}\makebox(0,0)[lb]{\smash{$\Omega_2$}}}%
    \put(0.2454387,0.12218202){\color[rgb]{0,0,0}\makebox(0,0)[lb]{\smash{$\Omega_3$}}}%
  \end{picture}%
\endgroup%
\caption{The set $\Omega$ in Example~\ref{r:es2}}
\label{fig:nonuniq2}
\end{figure}
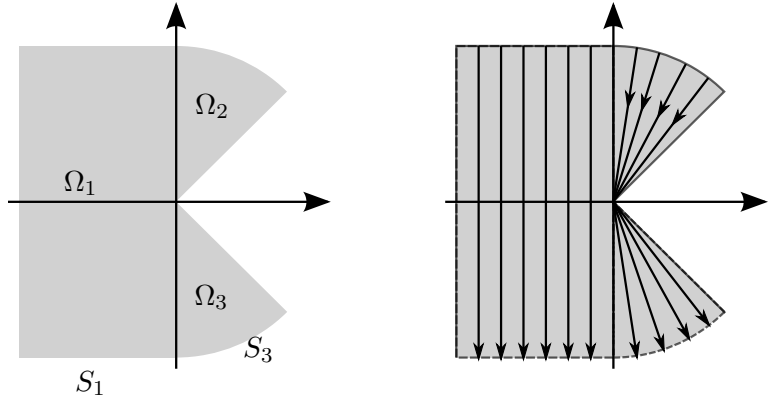
Let $\phi\colon\partial\Omega\to\R$ be the function
defined by
\[
\phi(x)=
\begin{cases}
0, & x\in S_1\cup S_3, \\
 +\infty, & \text{otherwise,}
 \end{cases}
 \]
where $S_1:=[-1,0]\times\{-1\}$ and
$S_3:=\{(\cos\theta, \sin\theta):\ -\pi/2 < \theta \leq -\pi/4\}$.
Let us consider the isotropic case ($\rho(\xi) = |\xi|$)
with a constant source $f\equiv 1$.

In this example the multiplicity of $v$-components
depends on the fact that geodesics are not forward unique.
%(i.e., they can branch).
Namely, given a point $y\in S_3$ and a point 
$z\in\Omega_2$,
the curve
$\gamma:= \cray{y,0}\cup\cray{0,z}$ is a geodesic.
It is then clear that we have a lot of geodesics branching
at $0$,
so that the mass collected in the region
$\Omega_2$ is transported to the origin,
and from the origin it can be distributed 
in infinitely many ways to any point of $S_3$.
\end{ehse}

In order to exclude the phenomena depicted in
Examples~\ref{r:es1} and \ref{r:es2},
we propose the following additional geometric assumption:
\begin{itemize}
\item[(H5)]
For every $x \in \Omega$ and for every $y\in\Pi(x)$ the segment
$\oray{y,x}$ is contained in $\Omega$.
\end{itemize}

As a consequence of (H5), every geodesic through a point $x\in\Omega$ is maximal and its support is a segment.
In particular, the geodesics cannot bifurcate
(i.e., they are non-branching in the interior).

%Indeed, we shall see that (H5) plays a fundamental role
%in proving the uniqueness of the transport density
%(see Lemma~\ref{l:uniq} below).

Under the assumption (H5), the Lax-Hopf function can be written as
\begin{equation}\label{f:LH2}
\uo(x) := \min\left\{\phi(y)+\pgauge(x-y):\ y\in\partial\Omega\right\},\quad
x\in \overline{\Omega},
%\uo(x) := \min\left\{\phi(y)+\pgauge(x-y):\ y\in\gaph\right\},\quad
%x\in \overline{\Omega},
\end{equation}
while
\begin{gather*}
\gaph =\{y\in \partial\Omega\colon\ \exists \ x\in \Omega\ \text{s.t.}\  
\oray{y,x}\subset\Omega\ \text{and}\ 
 \uo(x)=\phi(y)+\pgauge(x-y) \}
\,,
%\label{f:gaph}
\\
%\gaf =\{y\in \partial\Omega\colon\ \exists \ x\in\overline{\Omega}\
%\text{s.t.}\ \uo(x)=\phi(y)+\pgauge(x-y),\
%\oray{y,x} \cap \essspt f \neq\emptyset \}
\gaf =\{y\in \partial\Omega\colon\ \exists \ x\in\essspt f\
\text{s.t.}\ \oray{y,x}\subset\Omega\ \text{and}\ \uo(x)=\phi(y)+\pgauge(x-y) \}
\,.
%\label{f:gaf}
\end{gather*}

For every $x\in\Omega$, let $\Delta(x)$ be the set of directions through $x$
\[
\Delta(x):=\left\{\frac{x-y}{\pgauge(x-y)}:\ y\in\Pi(x)\right\}\,, \qquad x\in\Omega\,,
\]
where $\Pi(x)$ is the set of projections defined in (\ref{f:proj}) that,
under assumption (H5), can be written as
\[
\Pi(x)=\{y\in \gaph\colon\ \uo(x)=\phi(y)+\pgauge(x-y)\}\,.
\]

Let $D\subset\Omega$ be the set of those points with multiple projections, that is
%\begin{equation}\label{f:D}
\[
D:=\{x\in \Omega:\ \Delta(x)\ \text{is not a singleton}\},
\]
%\end{equation}
and for every $x\in \Omega\setminus D$,
let $p(x)$ and $d(x)$ denote
the unique elements in $\Pi(x)$ and $\Delta(x)$ respectively, i.e.
\begin{equation}
\label{f:pid}
\{p(x)\} = \Pi(x),\quad
\{d(x)\} = \Delta(x),\qquad
x\in\Omega\setminus D.
\end{equation}
It can be easily checked that $\uo$ grows linearly along every segment
joining $x\in \Omega$ to $y\in \Pi(x)$. Let us denote by
$b(x)$ be the normal distance from the set $D$, defined by
%\begin{equation}\label{f:tau}
\[
b(x):=
\begin{cases}
\sup\{t\geq 0;\ \uo(x+sd(x)))=\uo(x)+s,\ \forall s\in [0,t]\}, &
x\in \Omega\setminus D\,, \\
0 & x\in D\,,
\end{cases}
\]
%\end{equation}
and let $J$ be the set 
%\begin{equation}\label{f:qj}
\[
J:=\bigcup_{x\in\Omega}q(x),\quad q(x):=x+b(x)d(x),
\]
%\end{equation}
where we understand that $q(x)=x$ if $x\in D$.
\begin{dhef}[Transport ray]
We shall call \textsl{transport ray} through $x\in\Omega$
any segment $\cray{p,q(x)}$, $p\in\proj(x)$.
If $\cray{p,q}$ is a transport ray, the points $p$ and $q$ will be called
respectively the \textsl{initial} and the \textsl{final point} of the ray.
\end{dhef}
It is clear from the definition that, if $x\in \Omega\setminus D$, then
there is a unique transport ray $\cray{p(x),q(x)}$ through $x$.
In this case there exists a unique number $a(x) \in (-\infty, 0)$
such that
\[
p(x) = x + a(x)\, d(x)\qquad (x\in\Omega\setminus D).
\]
Moreover, if assumption (H5) holds,
according to Definition~\ref{d:maxg} the segment
$\cray{p(x), x}$ is a maximal geodesic through $x$.
On the other hand, if $x\in D$, any segment $\cray{p,x}$, with
$p\in\proj(x)$, is a transport ray through $x$.

The transport rays correspond to the segments where $\uo$ grows
linearly with maximal slope, i.e.
\begin{equation}
\label{f:lin}
\uo(x+t\, d(x)) = \uo(x) + t\qquad
\forall x\in\Omega\setminus D,\
\forall t\in [a(x), b(x)]\,.
\end{equation}

Let $\Sigma$ be the set of those points where $\uo$ is not differentiable.
The relationships between the singular sets related to the problem are the following
(see \cite{Bi} and \cite[Prop.~6.4]{CMl}).

\begin{prop}\label{p:D}
The sets $\Sigma$, $D$ and $J$ have zero Lebesgue measure.
Moreover, $D\subset\Sigma$ and $D\subset J\subset\overline{D}$,
with possibly strict inclusions. In addition, if $\uo$ is differentiable at $x\in\Omega$, then $x$
has a unique projection and $\Delta(x)=\{D\gauge(D\uo(x))\}$.
\end{prop}

%%%%%%%%%%%%%%%%%%%%%%%%%%%%%

Under the assumptions (H1)--(H5) we will be able to construct
a mass transport density $\vf$ such that the pair
$(\uo, \vf)$ solves the system~\eqref{f:MK}
(see Section~\ref{s:exi}).
Unfortunately, these assumptions are not enough to ensure the 
uniqueness of the mass density $v$, due to the possibility of transporting
a fictitious amount of mass along those transport rays $\oray{p,q}$ with
both endpoints on $\partial{\Omega}$. The weak formulation of 
the continuity equation (Definition~\ref{d:sol}(iii)) prevents this possibility
only for $q\in \partial{\Omega}\setminus \gafc$. 
The following example shows that a single $q\in\gafc$ may be the final point of a set of rays covering a region with positive Lebesgue measure.

\begin{figure}
\centering
\def\svgwidth{10cm}
%\input{nonuniq3.pdf_tex}
%% Creator: Inkscape 0.48+devel, www.inkscape.org
%% Accompanies image file 'nonuniq3.pdf' (pdf, eps, ps)
\begingroup%
  \makeatletter%
  \providecommand\color[2][]{%
    \errmessage{(Inkscape) Color is used for the text in Inkscape, but the package 'color.sty' is not loaded}%
    \renewcommand\color[2][]{}%
  }%
  \providecommand\transparent[1]{%
    \errmessage{(Inkscape) Transparency is used (non-zero) for the text in Inkscape, but the package 'transparent.sty' is not loaded}%
    \renewcommand\transparent[1]{}%
  }%
  \providecommand\rotatebox[2]{#2}%
  \ifx\svgwidth\undefined%
    \setlength{\unitlength}{308.15035278bp}%
    \ifx\svgscale\undefined%
      \relax%
    \else%
      \setlength{\unitlength}{\unitlength * \real{\svgscale}}%
    \fi%
  \else%
    \setlength{\unitlength}{\svgwidth}%
  \fi%
  \global\let\svgwidth\undefined%
  \global\let\svgscale\undefined%
  \makeatother%
  \begin{picture}(1,0.44804459)%
    \put(0,0){\includegraphics[width=\unitlength]{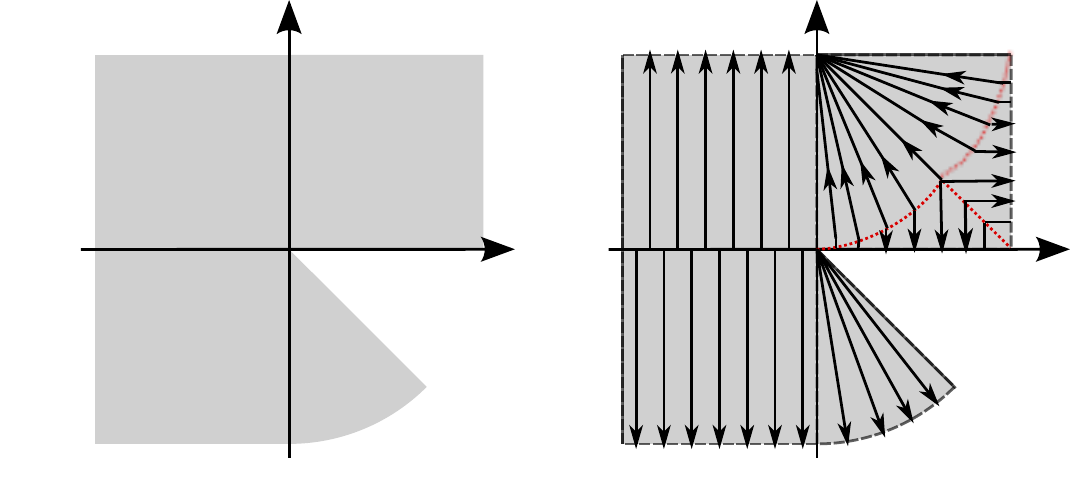}}%
    \put(0.32487796,0.3091701){\color[rgb]{0,0,0}\makebox(0,0)[lb]{\smash{$\Omega_2$}}}%
    \put(0.21808551,0.00734727){\color[rgb]{0,0,0}\makebox(0,0)[lb]{\smash{$S_1$}}}%
    \put(0.12722078,0.24099944){\color[rgb]{0,0,0}\makebox(0,0)[lb]{\smash{$\Omega_1$}}}%
    \put(0.29854134,0.09092701){\color[rgb]{0,0,0}\makebox(0,0)[lb]{\smash{$\Omega_3$}}}%
    \put(0.34789228,0.13715403){\color[rgb]{0,0,0}\makebox(0,0)[lb]{\smash{$S_3$}}}%
    \put(0.32193092,0.22801876){\color[rgb]{0,0,0}\makebox(0,0)[lb]{\smash{$S_2$}}}%
    \put(0.45173769,0.26696079){\color[rgb]{0,0,0}\makebox(0,0)[lb]{\smash{$S_2$}}}%
    \put(0.29596957,0.40974823){\color[rgb]{0,0,0}\makebox(0,0)[lb]{\smash{$S_4$}}}%
    \put(0.12722078,0.40974823){\color[rgb]{0,0,0}\makebox(0,0)[lb]{\smash{$S_1$}}}%
    \put(-0.00258599,0.21503809){\color[rgb]{0,0,0}\makebox(0,0)[lb]{\smash{$S_5$}}}%
  \end{picture}%
\endgroup%\caption{The set $\Omega$ in Example~\ref{r:es3}}
\label{fig:nonuniq3}
\end{figure}

\begin{ehse}
\label{r:es3}
Let $\Omega:=\Omega_1 \cup\Omega_2\cup\Omega_3 \subset \R^2$,
where
\begin{gather*}
\Omega_1  :=(-1,0)\times(-1,1),\qquad \Omega_2   := [0,1)\times(0,1), \\
\Omega_3  :=\left\{(r \cos\theta, r \sin\theta):\ r\in (0,1), \ 
-\frac{\pi}{2}<\theta<-\frac{\pi}{4}\right\}\,,
\end{gather*}
and let $\phi\colon\partial \Omega \to \R$ be a function such that, in the isotropic case with source $f=1$,  the geometry of transport rays is the one depicted in Figure~\ref{fig:nonuniq3}.
In particular $\gaph=\gaf=S_1\cup S_2$, where
\begin{gather*}
S_1 :=\left((-1,0]\times \{-1,1\}\right )\cup \{(\cos \theta, \sin\theta): -\frac{\pi}{2}<\theta<-\frac{\pi}{4}\}\,.  \\
 S_2 :=\left((0,1)\times\{0\}\right)\cup \left(\{1\}\times(0,1)\right)
\end{gather*}
(see the details in Section~\ref{s:examples}).
Notice that $(0,0)\in\gaphc$ is the final point of all the transport rays covering $\Omega_3$. Hence, we can construct a  $w\in L^1_+(\Omega_3)$ such that $\dive(w\, D\uo)=0$, which can be added, once prolonged to zero in $\Omega\setminus \Omega_3$, to any admissible  $v$, loosing the uniqueness of mass transport density.
\end{ehse}

In order to exclude the behaviour described in Example~\ref{r:es3}
we need the following additional assumption:
\begin{itemize}
\item[(H6)]
$\left(\overline{J\cap\partial\Omega}\right) \cap \gaphc = \emptyset$.
\end{itemize}
It is worth to remark that, as a consequence of (H5),
one has $(J\cap\partial\Omega)\cap\gaph = \emptyset$
(see Lemma~\ref{l:uniq} below), 
%which is equivalent to $(J\cap\partial\Omega)\cap\gaph=\emptyset$,
so that
(H6) can be viewed as a mild
additional assumption in order to separate
the sets $J\cap\partial\Omega$ and $\gaph$.

%We need to know that the final point of a transport ray
%cannot be the initial point of another transport ray.

\begin{lemma}\label{l:uniq}
Assume that $(H1)$, $(H2)$, $(H4)$ and $(H5)$ hold, and let
$\cray{p_1, q_1}$, $\cray{p_2, q_2}$ be two distinct non-trivial
transport rays.
Then $q_1\neq p_2$.
In other words, $J\cap\gaph = \emptyset$.
\end{lemma}

\begin{proof}
Assume by contradiction that $q_1 = p_2$,
and let $\gamma\in\ccurve{p_1,q_2}$ be the curve whose
support is $\cray{p_1, p_2}\cup\cray{p_2, q_2}$.
%Since $\cray{p_1, p_2}$ and $\cray{p_2, q_2}$ are two distinct non-trivial
%transport rays, the support of $\gamma$ is not a segment.
By \eqref{f:lin} we get
\[
\uo(q_2) = \uo(p_2) + \pgauge(q_2-p_2)
= \uo(p_1) + \pgauge(p_2-p_1) + \pgauge(q_2-p_2)
= \uo(p_1) + \len{\gamma},
\]
hence the curve $\gamma$ is a geodesic,
in contradiction with (H5).
\end{proof}

%%%%%%%%%%%%%%%%%%%%%%%%%%%%%%%%%%%%%%%%%%%%%%%%%%%%%%%%%%%%%%%%%%%%%%%%%%%%%%%%%%%%

\section{Existence of solutions}
\label{s:exi}

For the reader's convenience we collect here all the assumptions
that have been introduced in the previous section.
\begin{itemize}
\item[(H1)] $\Omega$ is an open, bounded, connected subset of
$\R^n$ with Lipschitz boundary;
\item[(H2)] $\gauge$ is the gauge function of a convex set $K\subseteq \R^n$ satisfying (\ref{f:hypk});
\item[(H3)] $f$ belongs to $L^1_+(\Omega)$;
\item[(H4)] $\phi\colon \partial\Omega \to (-\infty,+\infty]$ is a l.s.c.\ function, $\phi\not\equiv +\infty$;
\item[(H5)]
For every $x \in \Omega$ and for every $y\in\Pi(x)$ the segment
$\oray{y,x}$ is contained in $\Omega$;
\item[(H6)]
$\left(\overline{J\cap\partial\Omega}\right) \cap \gaphc = \emptyset$.
\end{itemize}

\medskip
This section will be devoted to the proof of the following result.
\begin{teor}\label{t:eu1}
Let $(H1)$-$(H6)$ hold.
Then there exists a unique $\vf\in L^1_+(\Omega)$ such that the pair $(\uo,\vf)$ is a solution
to the PDEs system \eqref{f:MK}.
\end{teor}

This theorem gives a partial uniqueness result
for the transport density, since it states that the $v$-component
associated to the profile $u = \uo$ is unique.
We shall see in Section~\ref{s:uni}, Theorem~\ref{t:uniu},
that, indeed, $\vf$ is the only admissible $v$-component of every
solution $(u,v)$.

\smallskip
Since $\uo\in\Xf$, and $\gauge(D\uo) = 1$ a.e.~in $\Omega$, 
in order to prove the existence part of Theorem~\ref{t:eu1}
it is enough to show that
\begin{equation}\label{f:vfsol}
\int_\Omega \vf \pscal{D\gauge(D\uo)}{D\psi}\, dx = \int_\Omega f \psi \, dx\,, \qquad
\forall\psi\in\test.
%C^{\infty}_c(\R^n \setminus \gaph).
\end{equation}

This can be done with a minor effort, since it turns out that we can choose as $v$-component
the same function constructed in \cite{Bi,CMl} for the case of admissible boundary data,
that is the one that satisfies
\begin{equation}\label{f:vfsol2}
\int_\Omega \vf \pscal{D\gauge(D\uo)}{D\psi}\, dx = \int_\Omega f \psi \, dx\,, \qquad
\forall\psi\in C^{\infty}_c(\Omega).
\end{equation}

In order to write the explicit form of the function $\vf$,
we need the fundamental result concerning the disintegration of the Lebesgue measure
along the transport rays due to S.~Bianchini (see \cite[Thm.\ 5.8]{Bi}).
The proof of this formula is mainly based on the fact that
the divergence of the vector field $d$
defined in \eqref{f:pid}
is a locally finite Radon measure in $\Omega$.
Moreover, decomposing $\dive d$ into its absolutely continuous and
singular part (w.r.t.\ the Lebesgue measure),
$\dive d = (\dive d)_{ac}\Leb^n +(\dive d)_s$,
it turns out that $(\dive d)_s$ is a positive measure in $\Omega$
(see \cite{Bi}, Section 5).
These properties allow to describe the evolution of
the Hausdorff measure of the
$(n-1)$-dimensional sections of ``cylinders'' of transport rays
(see Remark~\ref{r:haus}).

\begin{teor}\label{t:bian2}
Let $(H1)$, $(H2)$ and $(H4)$ hold.
Then there exists a sequence $\{A_k\}_{k\in\N}$ of subsets of $\Omega$ with the following properties.
\begin{itemize}
\item[(i)] $A_k\subset (\Omega\setminus D) \cap \{\pscal{x}{e_{j_k}}=z_k\}$ for some $j_k\in\{1,\ldots,n\}$ and $z_k\in\R$.
\item[(ii)] $A_k$ is measurable w.r.t.\ the $(n-1)$-dimensional Hausdorff measure.
\item[(iii)] The sets $T_k :=\bigcup_{x\in A_k}\oray{p(x),q(x)}$, $k\in\N$,
are pairwise disjoint,  Lebesgue measurable subsets of $\Omega$, and
$\Leb^n(\Omega\setminus \bigcup_k T_k) = 0$.
\item[(iv)] For every $h\in L^1(\Omega)$,
the function 
$t\mapsto h(x+t d(x))\,\alpha(t d_{j_k}(x), x)$,
where
$d_{j_k}(x):=\pscal{d(x)}{e_{j_k}}$,
belongs to $L^1(a(x), b(x))$ for
%$\mathcal{H}^{n-1}$-a.e.\ 
every $x\in \bigcup_k A_k$.
Moreover, the disintegration formula
\[
\int_\Omega h\, dx
 = \sum_k \int_{A_k}\left(\int_{a(x)}^{b(x)}h(x+t d(x))
 \,\alpha(t d_{j_k}(x), x)\,dt \right)
d_{j_k}(x)\, d \mathcal{H}^{n-1}(x)
\]
holds, and
for $\mathcal{H}^{n-1}$-a.e. $x\in \bigcup_k A_k$ the function
$\alpha(\cdot, x)$ is the solution
of the linear ODE
\begin{equation}\label{f:alpha}
\begin{cases}
\dfrac{d}{dt}\alpha(t d_{j_k}(x), x) =
(\dive d)_{ac}(x+td(x))\,\alpha(t d_{j_k}(x), x),
& t\in (a(x) , b(x))\,, \\
\alpha(0, x)=1\,.
\end{cases}
\end{equation}
Moreover, this solution is strictly positive.
\end{itemize}
\end{teor}

%\begin{proof}
%See \cite{Bi}, Theorem 5.8.
%\end{proof}

We remark that the decomposition introduced in 
Theorem~\ref{t:bian2} is clearly not unique.
In the following we shall always assume that such a decomposition has been fixed.
Once the decomposition is given, we shall use the notation
\[
\alpha(x+t d(x)) \equiv \alpha(t d_{j_k}(x), x)\, d_{j_k}(x),\qquad
k\in\N,\ x\in A_k,\ t\in (a(x), b(x)),
\]
so that $\alpha$ is defined and strictly positive on the set
\begin{equation}
\label{f:omegap}
\Omega' := \left\{
x + td(x):\ x\in\bigcup_k A_k;\ t\in (a(x), b(x))
\right\}
= \bigcup_k T_k
\end{equation}
i.e., by Theorem~\ref{t:bian2}(iii), almost everywhere on $\Omega$.
For every $x\in\Omega'$ the ODE \eqref{f:alpha}
along transport rays can be written as
\[
\frac{d}{dt} \alpha(x + td(x)) = (\dive d)_{ac}(x+td(x))\,
\alpha(x+t d(x)),
\]
and $\alpha = 1$ on $\bigcup_k A_k$,
while the disintegration formula for the Lebesgue measure given
in Theorem~\ref{t:bian2}(iv) simply becomes
\begin{equation}
\label{f:disint}
\int_\Omega h\, dx
 = \sum_k \int_{A_k}\left(\int_{a(x)}^{b(x)}h(x+t d(x))\alpha(x+t d(x))\,dt \right)
\, d \mathcal{H}^{n-1}(x)\,.
\end{equation}

%%%%%%%%%%%%%%
\begin{rk}\label{r:haus}
It can be of interest to understand the
geometrical meaning of the function $\alpha(t,x)$
in Theorem~\ref{t:bian2}(iv).
For fixed $k\in\N$ and $t\in\R$, let us define
the map
\[
\varphi^t\colon A_k\to \R^n,
\qquad x\mapsto \varphi^t(x) := x+t\, \frac{d(x)}{d_{j_k}(x)}\,,
\]
and let $A_k^t := \varphi^t(A_k)$.
Since $A_k$ is contained in the hyperplane
$\{x:\ \pscal{x}{e_{j_k}} = z_k\}$, we have that
$A_k^t\subset \{x: \pscal{x}{e_{j_k}} = z_k + t\}$.
Let us consider the measure
$\mu$ on $A_k$ such that
the push-forward $(\varphi^t)_{\sharp}\mu$ of $\mu$ is
$\mathcal{H}^{n-1}\lfloor A_k^t$, that is,
%$(\varphi^t)_{\sharp}\mu$
%is the measure on $A_k^t$ defined by
\begin{equation}
\label{f:mu}
\mu\left((\varphi^t)^{-1}(F)\right) =
\mathcal{H}^{n-1}(F)
\qquad
\text{for every}\ \mathcal{H}^{n-1}\text{-measurable set}\
F\subseteq A_k^t.
\end{equation}
%[(\varphi^t)_{\sharp}\mu](F) = 
%for every $\mathcal{H}^{n-1}$-measurable set $F\subseteq A_k^t$.
It turns out that 
the measure $\mu$ defined by \eqref{f:mu} 
is absolutely continuous w.r.t.\
$\mathcal{H}^{n-1}\lfloor A_k$  and
$\mu = \alpha(t,x) \mathcal{H}^{n-1}\lfloor A_k$
(see \cite[Lemma~5.4]{Bi}).
%Moreover, \eqref{f:alpha} follows from a divergence formula

%%%%%%%%%%%%%

Let $C\subset\Omega$ be a ``cylinder'' of rays of the form
\[
C = \left\{x + s d(x):\ x\in A_k, \ s\in (a(x),b(x))\right\}\,,
\]
and let $\chi_C$ denote its characteristic function. 
Then, for every $h\in L^1(C)$, using Fubini's theorem 
and the definition of $\mu$ we have that
\[
\begin{split}
\int_C h\, d\Leb^n & =
\int_{\R} \int_{A_k^t} h\, \chi_C\, d\mathcal{H}^{n-1}\, dt
\\ & =
\int_{\R} \int_{A_k^t}
(h\, \chi_C)(\varphi^t(x)) \alpha(t,x)\, d\mathcal{H}^{n-1}(x)\, dt
\\ & = \int_{A_k} \int_{a(x)}^{b(x)}
h(x+s\, d(x)) \alpha(s d_{j_k}(x), x) d_{j_k}(x)\, ds\, d\mathcal{H}^{n-1}(x)\,.
\end{split}
\]
Hence, roughly speaking, the disintegration formula (\ref{f:disint}) corresponds to covering
$\Omega$ with cylinders of transport rays and applying Fubini's theorem on each cylinder.
\end{rk}
%%%%%%%%%%%%%%%

We are now in a position to show that the problem of finding the solution $v$ to the equation
\begin{equation}\label{f:vpb}
\int_\Omega v \pscal{D\gauge(D\uo)}{D\psi}\, dx = \int_\Omega f \psi \, dx\,, \qquad
\forall\psi\in\test
%C^{\infty}_c(\R^n \setminus \gaph),
\end{equation}
can be solved by the Method of Characteristics.
We start recalling another result proved in \cite[Sect.~7]{Bi}.

\begin{teor}\label{f:bian3}
Let $(H1)$--$(H4)$ hold.
If $v\in L^1_+(\Omega)$ satisfies
\begin{equation}\label{f:vpb2}
\int_\Omega v \pscal{D\gauge(D\uo)}{D\psi}\, dx = \int_\Omega f \psi \, dx\,, \qquad
\forall\psi\in C^{\infty}_c(\Omega),
\end{equation}
then for a.e.\ $x\in\Omega$ the function $v$ is locally absolutely continuous along the ray
$t\mapsto x+t d(x)$, $t\in (a(x),b(x))$, and satisfies
\begin{equation}\label{f:ode}
\frac{d}{dt}\left[
v(x+t d(x))\, \alpha(x+t d(x))\right]
= - f(x+t d(x))\, \alpha(x+t d(x))\,,
\quad
t\in (a(x),b(x)).
\end{equation}
Moreover, if the final point $q(x) := x + b(x) d(x)$ belongs to $\Omega$, then
\begin{equation}\label{f:vfbc}
\lim_{t\to b(x)^-} v(x+t d(x))\, \alpha(x+t d(x)) = 0.
\end{equation}
Finally, the function $\vf\colon\Omega\to\R$
defined by
\begin{equation}\label{f:vf}
\vf(x) :=
\int_0^{b(x)} f(x+t d(x))\, \frac{\alpha(x+t d(x))}{\alpha(x)}\, dt,
\qquad\text{a.e.}\ x\in \Omega\,,
\end{equation}
belongs to $L^1_+(\Omega)$
and is a solution to~\eqref{f:vpb2}.
\end{teor}

\begin{rk}\label{r:vf}
The fact that $\vf$ satisfies (\ref{f:ode}) along almost every ray
can be easily proved observing that, for every $x\in\Omega'$,
\[
\vf(x + t d(x)) =
\int_t^{b(x)} f(x+s d(x))\, \frac{\alpha(x+s d(x))}{\alpha(x+t d(x))}\, ds,
\]
so that
\[
(\vf\,\alpha)(x + t d(x)) =
\int_t^{b(x)} (f\,\alpha)(x+s d(x))\, ds.
\]
{}From this last equality we also see that $\vf$ satisfies the terminal
condition (\ref{f:vfbc}) for
%almost every $x\in\Omega$,
every $x\in\Omega'$,
and not only for those points $x\in\Omega$ satisfying $q(x)\in\Omega$.
\end{rk}

\begin{rk}
\label{r:sptv}
{}From the definition \eqref{f:vf} of $\vf$ and the fact that
$\alpha(\cdot, x)$ is strictly positive in $(a(x), b(x))$,
we deduce that the essential support of $\vf$ coincides with
the closure of the set
\[
\left\{\bigcup \oray{p(x), x}:\ x\in\essspt f\right\}.
\]
In particular $\vf=0$ along all the transport rays starting from the points
of $\gaph \setminus \gaf$.
\end{rk}

\begin{rk}
Notice that neither (H5) nor (H6) are needed in order to find 
a solution to \eqref{f:vpb2}.
Namely, the weak formulation \eqref{f:vpb2} corresponds to the
continuity equation of the optimal mass transport problem in which
the mass can freely flow away from $\Omega$
at the first time it touches the boundary.
Hence, the phenomena depicted in Examples~\ref{r:es1} and
\ref{r:es2} 
%and \ref{r:es3} 
cannot happen.
\end{rk}

Since all functions $\psi\in C^{\infty}_c(\Omega)$ (extended to $0$ in $\R^n\setminus \Omega$) are admissible in the
weak formulation (\ref{f:vpb}) of the transport equation, it is clear
that every solution $v$ to (\ref{f:vpb}) satisfies (\ref{f:ode})
along almost every ray.

Using the disintegration formula for the Lebesgue measure it is
not difficult to prove that the function $\vf$ defined in
\eqref{f:vf} above is a solution to \eqref{f:vfsol},
provided that the mass can flow away at the initial points
%$p(x)$, $x\in\Omega$, i.e.
of the transport rays, i.e.
\[
\left\{p(x):\
x\in\bigcup_k A_k,\
\oray{p(x),q(x)}\cap\spt(f) \neq \emptyset
\right\}
\subseteq \gaf.
\]
This condition is clearly satisfied if (H5) holds.
%(see Theorem~\ref{f:bian}).

\begin{teor}\label{t:biane}
Assume that $(H1)$-$(H5)$ hold.
Then $(\uo, \vf)$ is a solution to \eqref{f:vfsol2}.
\end{teor}

\begin{proof}
{}From Theorem~\ref{f:bian3}, the function $\vf$ satisfies
the ODE \eqref{f:ode} along almost every ray
$x+t d(x)$, $t\in (a(x), b(x))$.
Moreover, as already observed in Remark~\ref{r:vf},
\begin{equation}
\label{f:exibd}
(\vf\alpha)(q(x)) := \lim_{t\to b(x)^-} (\vf\alpha)(x+td(x)) = 0,
\qquad
\mathcal{H}^{n-1}\text{-a.e.}\ x\in \bigcup_k A_k.
\end{equation}
It is not restrictive to assume
that \eqref{f:ode} and \eqref{f:exibd} hold for every $x$ belonging to the set $\Omega'$
defined in \eqref{f:omegap}.

Since, by Proposition~\ref{p:D},
$\Omega'\subseteq \Omega\setminus J \subseteq \Omega\setminus D$,
we have that, for every $x\in\Omega'$,
$d(x) = D\gauge(D\uo(x))$ and,
for $\psi\in\test$,
\[
\pscal{D\gauge(D\uo))(x+td(x))}{D\psi(x+td(x))}=
\pscal{d(x)}{D\psi(x+td(x))}=
\psi'(x+td(x))\,,
\]
where the prime denotes differentiation w.r.t.\ $t$.
Hence the disintegration formula 
\eqref{f:disint} gives
\begin{equation}\label{f:disxe}
\begin{split}
& \int_\Omega \vf \pscal{D\gauge(D\uo)}{D\psi}\, dx  \\
& = \sum_k \int_{A_k}\int_{a(x)}^{b(x)}
(\vf \alpha)(x+td(x))\psi'(x+td(x))\,dt
\, d \mathcal{H}^{n-1}(x)\,.
\end{split}
\end{equation}
An integration by parts in the inner integral of \eqref{f:disxe}
and formula \eqref{f:ode}
% and \eqref{f:zerob}
lead to
\begin{equation}
\label{f:disxf}
\begin{split}
& \int_{a(x)}^{b(x)}  (\vf\alpha)(x+td(x))\psi'(x+td(x))\,dt \\ & =
(\vf\alpha\psi)(q(x)) - (\vf\alpha\psi)(p(x)) +
\int_{a(x)}^{b(x)} (f\alpha)(x+td(x))\psi(x+td(x))\,dt\,,
\end{split}
\end{equation}
where we use the convention
\[
%(v\alpha)(q(x)) := \lim_{t\to b(x)-} (v\alpha)(x+t d(x)),
%\quad
(\vf\alpha)(p(x)) := \lim_{t\to a(x)+} (\vf\alpha)(x+t d(x)).
\]
If $p(x)$ belongs to $\gafc$ then
the test function $\psi$ vanishes in a neighbourhood of $p(x)$.
On the other hand, if $p(x)\not\in\gafc$, then
$f$ vanishes on the ray $\oray{p(x),q(x)}$, so that,
by \eqref{f:ode},
$\vf\alpha$ is constant along that ray.
Since $(\vf\alpha)(q(x)) = 0$, then also
$(\vf\alpha)(p(x)) = 0$. 

Finally, using again the disintegration formula,
the weak formulation \eqref{f:vpb}, \eqref{f:exibd}, 
and the boundary conditions 
$(\vf\alpha\psi)(p(x)) = (\vf\alpha)(q(x)) = 0$,
we conclude that $\vf$ is a solution to \eqref{f:vpb},
i.e.\ the pair $(\uo, \vf)$ is a solution to \eqref{f:MK}.
\end{proof}

%%%%%%%%%%%%%%%%%%%%%%%%%%%%%%%%%%%%%%%%%%%%%%%%%%%%%%%%%%

In order to prove a uniqueness result we need
to exploit assumption (H6).

\begin{teor}\label{f:bian}
Assume that $(H1)$-$(H6)$ hold.
Then a function $v\in L^1_+(\Omega)$ satisfies \eqref{f:vpb}
if and only if,
for a.e.\ $x\in\Omega\setminus D$, the function $v$
is locally absolutely continuous along the ray
$t\mapsto x+t d(x)$, $t\in (a(x),b(x))$ and
\begin{equation}\label{f:odev}
\begin{cases}
\displaystyle
\frac{d}{dt}\left[
v(x+t d(x))\, \alpha(x+t d(x))\right]
= - f(x+t d(x))\, \alpha(x+t d(x)),\\
\displaystyle
\lim_{t\to b(x)^-} v(x+t d(x))\, \alpha(x+t d(x)) = 0.
\end{cases}
\end{equation}
As a consequence, the function $\vf$ defined in \eqref{f:vf}
is the unique solution of~\eqref{f:vpb}.
\end{teor}

\begin{proof}
By Theorem~\ref{t:bian2}(iii)
it is clear that two functions both satisfying
(\ref{f:odev}) along almost every ray
must coincide almost everywhere.
Hence the uniqueness result will be achieved once we prove that
every solution to \eqref{f:vpb} satisfies (\ref{f:odev}).

Let $v\in L^1_+(\Omega)$ be any solution to \eqref{f:vpb}, and
let $\psi\in\test$.
Notice that $v$ is a solution to \eqref{f:vpb2};
by Theorem~\ref{f:bian3}, it is not restrictive to assume
that \eqref{f:ode} holds for every $x$ belonging to the set $\Omega'$
defined in \eqref{f:omegap},
whereas \eqref{f:vfbc} holds for every $x\in\Omega'$ such that
$q(x)\in\Omega$.

Reasoning as in the proof of Theorem~\ref{t:biane}
(see \eqref{f:disxe} and \eqref{f:disxf} with $v$ instead of $\vf$),
and using the fact that $v$ is a solution to \eqref{f:vpb},
we obtain
\[
\sum_k \int_{A_k} [(v\alpha\psi)(q(x))-(v\alpha\psi)(p(x))]
\,d \mathcal{H}^{n-1}(x) = 0\,.
\]
By \eqref{f:vfbc}, if $q(x)$ belongs to $\Omega$,
then $(v\alpha)(q(x)) = 0$.
Moreover, if $p(x)$ belongs to $\gafc$ then
the test function $\psi$ vanishes in a neighbourhood of $p(x)$.
Hence the formula above is equivalent to
\begin{equation}\label{f:nullbca}
\sum_k \int_{B_k} (v\alpha\psi)(q(x))\,d \mathcal{H}^{n-1}(x) 
- \sum_k \int_{C_k} (v\alpha\psi)(p(x))\,d \mathcal{H}^{n-1}(x) 
= 0\,,
\end{equation}
where
\[
B_k := \{x\in A_k:\  q(x)\in\partial\Omega\}\,,\qquad
C_k := \{x\in A_k:\  p(x)\not\in\gafc\}\,.
\]
Let us consider the following subsets of $\partial\Omega$:
\[
\mathcal{B} := \left\{ q(x):\ x\in \bigcup_k B_k\right\}  
\subseteq J \cap \partial\Omega,
\qquad
\mathcal{C} := \left\{ p(x):\ x\in \bigcup_k C_k\right\}
\subseteq \gaph\setminus\gafc\,.
\]
{}From %(H5) (in the form of Lemma~\ref{l:uniq}) and 
the very definition of $\mathcal{C}$, (H6) and
the inclusion $\gaf\subseteq\gaph$
we deduce that
\begin{equation}
\label{f:sep}
\overline{\mathcal{B}} \cap \gafc = \emptyset, \qquad
\overline{\mathcal{B}}\cap \overline{\mathcal{C}} = \emptyset.
\end{equation}
Hence, choosing $\delta > 0$ such that
$B_{\delta}(\mathcal{B}) \cap \mathcal{C} = B_{\delta}(\mathcal{B}) \cap \gaf = \emptyset$, 
by \eqref{f:sep} we can construct a function
$\psi\in C^{\infty}_c(\R^n)$ 
satisfying
\[
\psi > 0\quad\text{on}\ \mathcal{B}, \qquad
\psi = 0\quad\text{on}\ [\R^n\setminus B_{\delta}(\mathcal{B})]
\supset \gaf \cup \mathcal{C}\,,
\]
so that $\psi\in\test$.
Recalling that $v\alpha\geq 0$,
and using $\psi$ as a test function in \eqref{f:nullbca}
we conclude that
%$(v\alpha)(p(x)) = 0$ for $\mathcal{H}^{n-1}$-a.e.\ $x\in C_k$ and
$(v\alpha)(q(x)) = 0$ for $\mathcal{H}^{n-1}$-a.e.\ $x\in B_k$, $k\in\N$,
so that the boundary condition in \eqref{f:odev} holds.
\end{proof}

Another easy consequence of both the representation formula
for $\vf$ and the disintegration \eqref{f:disint} is the following stability
result.

\begin{teor}[Stability]
For every $f_1, f_2\in L^1_+(\Omega)$ there holds
%\begin{equation}\label{f:stab}
\[
\|v_{f_1} - v_{f_2}\|_{L^1(\Omega)} \leq
\diam(\Omega)\, \|f_1 - f_2\|_{L^1(\Omega)}\,,
\]
%\end{equation}
where $\diam(\Omega)$ denotes the diameter of the
set $\Omega$.
\end{teor}

\begin{proof}
{}From the disintegration formula \eqref{f:disint}, the representation formula \eqref{f:vf}
and Remark~\ref{r:vf}
we have that
\[
\begin{split}
\|v_{f_1} - v_{f_2}\|_{L^1(\Omega)} &=
\sum_k \int_{A_k}\left(\int_{a(x)}^{b(x)}
|(v_{f_1} - v_{f_2})(x+t d(x))|\,\alpha(x+t d(x))\,dt \right)
\, d \mathcal{H}^{n-1}(x)
\\ & \leq
\sum_k \int_{A_k}\left[\int_{a(x)}^{b(x)}
\left(\int_t^{b(x)}(|f_1 - f_2|\,\alpha)(x+s d(x))ds\right)
\,dt \right]
\, d \mathcal{H}^{n-1}(x)
\\ & \leq
\diam(\Omega)\sum_k \int_{A_k}
\left(\int_{a(x)}^{b(x)}(|f_1 - f_2|\,\alpha)(x+s d(x))ds \right)
\, d \mathcal{H}^{n-1}(x)
\\ & =
\diam(\Omega)\, \|f_1-f_2\|_{L^1(\Omega)}\,,
\end{split}
\]
completing the proof.
\end{proof}

%%%%%%%%%%%%%%%%%%%%%%%%%%%%%%%%%%%%%%%%
\section{Uniqueness of the solutions}
\label{s:uni}

Up to now we have proved that there exists a unique $v=\vf$ such that the pair
$(\uo,\vf)$ solves (\ref{f:MK}). In this section we shall prove that, actually,
$\vf$ is the unique admissible $v$--component for (\ref{f:MK}).

For what concerns the $u$ component, we start by proving that the Monge--Kantorovich system (\ref{f:MK})
is the Euler-Lagrange condition for the minimum problem
with gradient constraint
\begin{equation}\label{f:min}
\min \left\{-\int_{\Omega}fu\, dx:\ u\in \Xf\right\}.
\end{equation}
It is clear that, since $f$ is non-negative and $\uo$ is the maximal
element in $\Xf$, then $\uo$ is a solution to (\ref{f:min}). Moreover
every $u$ which minimizes (\ref{f:min}) has to agree with $\uo$ on $\essspt f$.

\begin{teor}\label{t:equiv}
Assume that $(H1)$--$(H5)$ hold.
The minimum problem \eqref{f:min} and the system of PDEs \eqref{f:MK}
are equivalent in the following sense.
\begin{itemize}
\item[(i)]
$u\in\Xf$ is a solution to \eqref{f:min} if and only if
there exists $v \in L^1_+(\Omega)$
such that
$(u,v)$ is a solution to~\eqref{f:MK}.
\item[(ii)]
Let $u\in\Xf$ be a solution to~\eqref{f:min}.
Then $(u,v)$ is a solution to~\eqref{f:MK}
if and only if $(\uo, v)$ is a solution to~\eqref{f:MK}.
\end{itemize}
\end{teor}

\begin{proof}
The proof is similar to the one of Theorem~5.3 in
\cite{CMl}; for the reader's convenience we sketch here
the main steps.

In what follows we will freely use that fact that,
by a density argument, the difference $u-w$ of two functions
$u,w\in\Xf$ can be used as test function in the weak formulation
\eqref{f:vpb} of the transport equation.

Let us denote by $I_{K}$ the indicator function of the set $K$, that is
%\begin{equation}\label{f:IK}
\[
I_K(x) :=
\begin{cases}
0 & \textrm{if} \ x\in K\,, \\
+\infty & \textrm{if} \ x\in \R^n \setminus K\,,
\end{cases}
\]
so that
\[
F(u):=\int_{\Omega}[I_K(Du)-fu]\, dx = -\int_{\Omega}fu\, dx,
\qquad \forall u\in \Xf\,.
\]
Since the gauge function $\gauge$ is differentiable in $\R^n \setminus \{0\}$,
the subgradient of $I_K$
can be explicitly computed, obtaining
%\begin{equation}\label{subindk}
\[
\partial I_{K}(\xi)=
\begin{cases}
\{\alpha D\gauge(\xi)\colon \alpha \geq 0\} & \text{if}\ \xi\in \partial K, \\
\emptyset & \text{if}\ \xi\not\in K, \\
\{0\} & \text{if}\ \xi\in \inte K
\end{cases}
\]
%\end{equation}
(see e.g.\ \cite[Sect.~23]{Rock}).

Hence, if $(u,v)$ is a solution to (\ref{f:MK}),
by condition (ii) in Definition~\ref{d:sol},
we have that $v(x) D\gauge(Du(x)) \in \partial I_{K}(Du(x))$ for a.e.\ $x\in\Omega$, so that,
for every $w\in \Xf$
%\begin{equation}\label{f:minprf}
\[
F(w) - F(u)\geq
\int_\Omega v \pscal{D\gauge(Du)}{Dw-Du}\, dx  -
\int_\Omega f\,(w-u)\, dx =0\,,
\]
%\end{equation}
where the last equality follows from
the fact that $w-u$ can be taken as test function
in~\eqref{f:vpb}.
%Corollary~\ref{c:density}.
This proves that $u$ is a solution to (\ref{f:min}).

Assume now that $u\in\Xf$ is a minimizer for $F$, so that $f(u-\uo)=0$
a.e.\ in $\Omega$, due to the maximality of $\uo$ in $\Xf$ and the fact that $f\geq 0$.
%By Corollary~\ref{c:density}
Again we can choose $u-\uo$ as test function in the transport equation \eqref{f:vpb} solved by $\uo$, getting
\[
0=\int_\Omega \vf \pscal{D\gauge(D\uo)}{Du-D\uo}\, dx =
-\int_\Omega \vf (1-\pscal{D\gauge(D\uo)}{Du})\, dx\,.
\]
On the other hand, by Theorem~\ref{t:sch}(ii)
and the fact that $Du\in K$ a.e.\ in $\Omega$, we have
\begin{gather*}
1-\pscal{D\gauge(D\uo(x))}{Du(x)} \geq 0,\ \text{a.e.}\ x\in \Omega\,, \\
1-\pscal{D\gauge(D\uo(x))}{Du(x)} = 0 \
\Longleftrightarrow\ D\gauge(D\uo(x))=D\gauge(Du(x)),
\end{gather*}
so that
$\vf D\gauge(D\uo)=\vf D\gauge(Du)$
and $\vf (1-\gauge(Du))=0$ a.e.\ in $\Omega$,
that is $(u, \vf)$ is a solution of (\ref{f:MK}).
This concludes the proof of (i).

Let us prove (ii).
The previous computation shows that if $(\uo,v)$ is a solution of \eqref{f:MK}, and $u\in\Xf$ is a solution of the minimum problem \eqref{f:min},
then also $(u,v)$ is a solution of \eqref{f:MK}.
Finally, let $(u,v)$ be a solution to \eqref{f:MK}.
Upon observing that $v\,\gauge(Du)=v$ a.e.\ in $\Omega$,
and choosing $u-\uo$
as test function in the weak formulation of the transport equation
\[
-\dive(v\, D\gauge(Du)) = f,
\qquad\textrm{in}\ \Omega\,,
\]
%\eqref{f:vpb}, 
we conclude that $(\uo,v)$ is a solution to \eqref{f:MK}.
\end{proof}

As a consequence of the previous results we obtain the following uniqueness result for the $v$ component.

\begin{cor}
Assume that $(H1)$--$(H6)$ hold.
If $(u,v)$ is a solution to~\eqref{f:MK}, then $v=\vf$ in $\Omega$, and $u=\uo$ on $\essspt f$.
\end{cor}

\begin{proof}
Let $(u,v)$ be a solution to~\eqref{f:MK}.
By Theorem~\ref{t:equiv}(i) we have that $u$ is a solution
to \eqref{f:min} hence,
by Theorem~\ref{t:equiv}(ii),
$(\uo, v)$ is a solution to~\eqref{f:MK}.
The conclusion now follows from Theorem~\ref{f:bian}.
\end{proof}

We now introduce another element of $\Xf$,
which will play the r\"ole of the minimal admissible profile.
This minimal profile $\uf$ depends on the source $f$ and
is defined by
\begin{equation}
\label{f:uf}
\uf(x) := \sup\{\uo(z) - L(\gamma):\
z\in\spt(f),\
\gamma\in\ccurve{x,z}\}\,,
\end{equation}
with the convention $\uf\equiv -\infty$ if $f\equiv 0$.
(We recall that, by definition, $\spt(f)$ is a relatively closed set in $\Omega$.)

\begin{prop}
\label{p:supp}
Assume that $(H1)$--$(H4)$ hold.
If $f\not\equiv 0$ then the function $\uf$ defined in
\eqref{f:uf} belongs to $\Xf$ and $\uf=\uo$ on $\essspt f$.
Moreover, every function $u\in\Xf$ such that
$u = \uf$ in $\essspt(f)$ satisfies $\uf\leq u\leq\uo$ in
$\overline{\Omega}$.
\end{prop}

\begin{proof}
By the very definition of $\uf$, we have that $\uf$ satisfies
\eqref{f:lipur} in $\Omega$, hence, by Lemma~\ref{l:diseqrho}, 
$\uf\in W^{1,\infty}(\Omega)$
and $D\uf\in K$ a.e.\ in $\Omega$.
Moreover, for every  $z\in\spt(f)$ and $y\in\partial\Omega$
we have that
\[
\phi(y)\geq\uo(y)\geq\uo(z) - L(\gamma),
\qquad \forall \gamma\in\ccurve{y,z},
\]
so that $\phi(y)\geq\uf(y)$.
In order to prove that
$\uf = \phi$ on $\gaf$,
%$\min_{\partial\Omega}(\phi-u) = 0$,
let $y\in\gaf$ and let $z\in\overline{\spt(f)}$ be such that
there exists a maximal geodesics $\gamma\in\ccurve{y,z}$,
i.e.\
$\uo(z) = \phi(y) + L(\gamma)$. Then
\[
\phi(y)\geq\uf(y)\geq\uo(z) - L(\gamma) = \phi(y).
\]
It remains to prove that, if $u\in\Xf$ coincides with $\uf$
on $\spt(f)$, then $u\geq\uf$.
(The inequality $u\leq\uo$ is trivially satisfied by the
maximality of $\uo$ in $\Xf$.)
Namely, for every $x\in\Omega$ there exist $z\in\overline{\spt(f)}$
and $\gamma\in\ccurve{x,z}$ such that
$\uf(x) = \uo(z) - L(\gamma)$, hence
by Lemma~\ref{l:diseqrho}(iii)
\[
\uf(x) + L(\gamma) = \uo(z) = u(z) \leq u(x) + L(\gamma)
\]
i.e.\ $\uf(x) \leq u(x)$.
\end{proof}

\begin{rk}
Since $\uf = \uo$ in $\essspt (f)$, exploiting
the explicit representation formula \eqref{f:vf} of $\vf$,
we can infer that $\uf = \uo$ in $\essspt(\vf)$
(see also Remark~\ref{r:sptv}).
\end{rk}

The following result is the analogous of
Theorem~7.2 in \cite{CMl}.

\begin{teor}[Uniqueness]\label{t:uniu}
Assume that $(H1)$--$(H6)$ hold.
Then a function $u\in\Xf$ is a solution to~\eqref{f:min}
if and only if $\uf\leq u\leq\uo$.
Moreover, the function $\uf$ coincides with $\uo$ in $\Omega$
(and hence $\uo$ is the unique solution to \eqref{f:min})
if and only if
$J \subseteq \overline{\spt(f)}$.
\end{teor}

\begin{proof}
Let us first consider the trivial case $f\equiv 0$.
In this case $\Xf = \spaceX$, $\uf \equiv -\infty$
and every function $u\in\spaceX$ is a solution to the minimum
problem~\eqref{f:min}.

Let now assume that $f\not\equiv 0$.
The first assertion follows from Proposition~\ref{p:supp}
and from the fact that $u\in\Xf$ is a solution
to \eqref{f:min} if and only if
$u=\uo$ on $\spt(f)$.

Let us prove the uniqueness result.

Let $J \subseteq \overline{\spt(f)}$.
{}From Proposition~\ref{p:supp} we have that
$\uf = \uo$ on $\overline{\spt(f)}$, hence on $J$.
Let $x\in \Omega \setminus J$ be given, and let $q(x)\in J$
be the endpoint of the ray through $x$. We have
\[
\uo(x)=\uo(q(x))-\pgauge(q(x)-x) =\uf(q(x))-\pgauge(q(x)-x)\leq \uf(x)\leq  \uo(x)\,,
\]
and hence $\uf(x)=\uo(x)$.

Assume now that $\uo = \uf$ in $\Omega$,
and assume, by contradiction, that there exists a point
$x_0\in J$, $x_0\not\in \overline{\spt(f)}$.
By definition, there exist $z\in \overline{\spt(f)}$
and a curve $\gamma\in\ccurve{x_0,z}$ such that
\begin{equation}
\label{f:uo1}
\uf(x_0) = \uo(z) - \len{\gamma}.
\end{equation}
%Let $(z_j)_j\subset\spt(f)\subseteq\Omega$ be a sequence
%converging to $z$.
%Let $y_j\in\proj(z_j)$; up to a subsequence, we can assume
%that $y_j\to y$, with $y\in\proj(z)$.
%By (H5), we have that $\oray{y_j, z_j}\subset\Omega$, hence
%$\cray{y,z}\subset\overline{\Omega}$, which yields
%\[
%\uo(z) = \uo(y) + \pgauge(z-y).
%\]
Moreover, by \eqref{f:LH2}, if $y_0\in\proj(x_0)$
one has
\begin{equation}
\label{f:uo2}
\uo(x_0) = \uo(y_0) + \pgauge(x_0-y_0).
\end{equation}
Since, by assumption, $\uo(x_0) = \uf(x_0)$,
\eqref{f:uo1} and \eqref{f:uo2} yield
\[
\pgauge(z-y_0) \leq \pgauge(x_0-y_0) + \len{\gamma}
= \uo(z) - \uo(y_0) \leq \pgauge(z-y_0),
\]
i.e., $\pgauge(x_0-y_0) + \len{\gamma} = \pgauge(z-y_0)$.
%Since $\len{\gamma}\leq\pgauge(z-x_0)$,
{}From Theorem~\ref{t:sch}(ii) it follows that
$\cray{y_0,x_0}\cup\gamma = \cray{y_0,z}$, that is
$\len{\gamma} = \pgauge(z-x_0)$ and
$x_0\in \cray{y_0, z}$.
Finally, since $z\neq x_0$, by \eqref{f:uo1} and the 
definition of final point
we have that $x_0\not\in J$, a contradiction.
\end{proof}

For the reader's convenience, we summarize here the results
we have obtained.
\begin{cor}
Assume that $(H1)$--$(H6)$ hold. Then
\begin{itemize}
\item[(i)] If $(u,v)$ is a solution to \eqref{f:MK}, then $v=\vf$.
\item[(ii)] $(u,\vf)$ is a solution to \eqref{f:MK} if and only if $u\in\spaceX$ and $\uf\leq u\leq \uo$ in $\Omega$.
\item[(iii)] $(\uo,\vf)$ is the unique solution to \eqref{f:MK} if and only if $J\subseteq \overline{\essspt f}$.
\end{itemize}
\end{cor}

%%%%%%%%%%%%%%%%%%%%%%%%%%%%%%%%%%%%%%%%%%%%%%%%%%%%%%%%%%
\section{The isotropic case: application to sandpiles} \label{s:sand}

In this section we shall focus our attention on the isotropic case $\gauge(\xi)=|\xi|$, $\xi\in \R^n$, ($n=2$ in the model problem) translating the results of the previous sections in terms of description of the
equilibrium configurations for sandpiles on a flat table with a vertical rim, and comparing these results with the ones known in literature.

In what follows we always assume that (H1), (H3), (H4), (H5) and (H6) hold. 
Furthermore, when considering the sandpile model, 
one could prefer not to allow profiles with negative height.
If this is the case,
one has to assume that 
\begin{itemize}
\item[(a)] The boundary datum $\phi$ (corresponding to the height of the rim) is non-negative;
\item[(b)] The space $\Xf$ contains only non-negative functions;
\item[(c)] The minimal function $\uf$, defined in \eqref{f:uf},
is replaced by $\max\{\uf,\, 0\}$.
\end{itemize}

If the sand is poured by a vertical source $f$ on the table occupying the region $\Omega$, the admissible equilibrium configurations can be depicted in terms of the pair $(u,v)$, the profile of the standing layer and
the thickness of the rolling layer, which are solutions of the PDEs system
(\ref{f:MKintro}), that is

\begin{itemize}
\item[($\text{i}'$)] $(u,v)\in \Xf \times L^1_+(\Omega)$;
\item[($\text{ii}'$)] $(1-|Du|)v=0$ a.e.\ in $\Omega$;
\item[($\text{iii}'$)] for every 
%$\psi\in C^{\infty}_c(\R^n \setminus \gaf)$
$\psi\in\test$
\[
\int_\Omega v \pscal{Du}{D\psi}\, dx = \int_\Omega f \psi \, dx\,,
\]
\end{itemize}
where
\[
\gaf=\{y\in \partial\Omega\colon\ \exists x \in \essspt f
\ \text{s.t.}\ \oray{y,x}\subset\Omega\
\text{and}\ \uo(x)=\phi(y)+|x-y|
\}
%\gaf=\{y\in \partial\Omega\colon\ \exists x \in \overline{\Omega}
%\ \text{such that } \uo(x)=\phi(y)+|x-y|, \oray{y,x}\cap \essspt f \neq %\emptyset\}
\]
is defined in term of the Lax-Hopf function
\[
\uo(x)=\min\{ \phi(y)+|x-y|,\ y\in\partial\Omega
%\ \text{such that } \oray{x,y}\subset \Omega
\}.
\]
Following \eqref{f:LH2}, we have that
\[
\uo(x)=\min\{\phi(y)+|x-y|,\ y\in\gaph\},
\qquad x\in\Omega,
\]
where
\[
\gaph=\{y\in \partial\Omega\colon\ \exists x \in\Omega\
\ \text{s.t.}\ \oray{y,x}\subset\Omega\
\text{and}\ \uo(x)=\phi(y)+|x-y|\}.
\]
Hence in the sandpiles model, $\gaph$ is the set of the initial points 
of the transport rays 
where the matter is allowed to run down, that is where the $v$--component could be non-zero.

On the other hand, $\gaf$ is the set of the initial points of the
transport rays along which the transport is active. It
is the effective border of the container, where the standing layer fills the gap with the rim of height $\phi$ (since $u=\uf=\uo=\phi$ on $\Gamma_f$ for every admissible profile, see Theorem \ref{t:uniu}) and the exceeding sand falls down.

Finally, $\partial\Omega \setminus \gaph$ is the part of the boundary closed by walls that the sand cannot overcome,
no matter what the source is.

The results obtained in the previous sections, interpreted in the
light of the sandpile model, are the following.
\begin{teor}
Assume that $(H1)$, $(H3)$--$(H6)$ hold. Then:
\begin{itemize}
\item[(i)] 
The thickness of the rolling layer is uniquely determined by
the data of the problem.
\item[(ii)] 
The admissible profiles of the standing layer are wedged between
the minimal profile $\uf$ and the maximal profile $\uo$.
In particular, every admissible profile has to agree with the
maximal one in the region where the transport is active.
\item[(iii)] 
There exists a unique admissible configuration if and only if
the source pours sand on the whole ridge of the maximal profile.
\end{itemize}
\end{teor}

These results extend the known results for the open table
problem without walls (see \cite{CaCa,CCCG,CCS,Pr})
and with walls (see \cite{CFV}).
Moreover, we have clarified the issue of the lack of
uniqueness of the rolling layer we have faced
in a less general setting in \cite{CMl}.

%%%%%%%%%%%%%%%%%%%%%%%%%%%%%%%%%%%%%%%%%%%%%%
\section{Examples}
\label{s:examples}

In this section we detail some computation
that were skipped while presenting
Examples~\ref{r:es1}, \ref{r:es2} and \ref{r:es3}
in Section~\ref{s:problem}.

In all examples the metric is isotropic,
i.e.\ $\gauge(\xi) = |\xi|$. 
Moreover, we assumed that $f\equiv 1$,
so that the set $\gaf$ (defined in \eqref{f:gafg})
of initial points of maximal
geodesics intersecting the support of $f$
coincides with the set $\gaph$ (defined in \eqref{f:gaphg}) 
of all initial points of maximal geodesics. 
In the first two examples only assumptions (H1)--(H4) hold,
so that maximal geodesics need not be segments.
In any case, by Proposition~\ref{p:supp} we have that
$\uf = \uo$, so that the Lax--Hopf function $\uo$ is the
unique candidate for the $u$-component of the solution.

Concerning the $v$-component,
by Theorem~\ref{f:bian3} we have that,
along almost every transport ray, the ODE~\eqref{f:ode}
is satisfied. 
Moreover, if the final point $q$ of a transport ray does not
belong to the closure of the set of initial points,
then $v\alpha$ must vanish on $q$
(see the end of the proof of Theorem~\ref{f:bian}).

\subsection*{Details on Example~\ref{r:es1}}

We have already observed that
every solution $(u,v)$ of \eqref{f:MK}
must satisfy $u=\uo$ in $\Omega$, whereas
$v(x) = 1-x_2$ for every $x = (x_1, x_2)\in \Omega_2$,
so that, for every test function
$\psi\in C^{\infty}_c(\R^2 \setminus \gafc)$,
\[
\int_{\Omega_2}\left( v\pscal{D\uo}{D\psi} - \psi\right) dx = 0.
\]
On the other hand, $v$ must be locally absolutely continuous along
any ray $(0,1)\ni r\mapsto (r\cos\theta, r\sin\theta)$,
$0<\theta<\pi/2$, and
\[
\frac{d}{dr} \left[r\, v(r\cos\theta, r\sin\theta)\right]
= r,
\qquad -r\in (0,1),\ 0<\theta<\frac{\pi}{2}
\]
(see \eqref{f:ode}).
Then, on $\Omega_1$, we have that
\[
v(r\cos\theta, r\sin\theta)
= \frac{1-r^2}{2r} + \frac{c(\theta)}{r}
\qquad r\in (0,1),\ 0<\theta<\frac{\pi}{2}
\]
for some integrable function $c\geq 0$.
Now, an explicit computation in polar coordinates gives
\[
%I_1 :=
\int_{\Omega_1}\left( v\pscal{D\uo}{D\psi} - \psi\right) dx
= \int_0^{\pi/2}
\left[c(\theta) \psi(\cos\theta,\sin\theta)
- \left(\frac{1}{2} + c(\theta)\right) \psi(0)
\right]\, d\theta
\]
and this last integral clearly cannot vanish for every choice
of the test function $\psi$.

%%%%%%%%%%%%%%%%
\subsection*{Details on Example~\ref{r:es2}}

The Lax-Hopf function can be easily computed:
\[
\uo(x) =
\begin{cases}
1 + x_2, &\text{if}\ x\in\Omega_1,\\
1 + |x|, &\text{if}\ x\in\Omega_2,\\
1 - |x|, &\text{if}\ x\in\Omega_3.
\end{cases}
\]
It is clear that 
$\gaph = \gaf = (S_1\cup S_3)\setminus\{(-1,-1),(2^{-1/2}, - 2^{-1/2})\}$.

We claim that every pair $(\uo, \vf + w)$,
with
\[
\vf(x) :=
\begin{cases}
1+x_2, &\text{if}\ x\in\Omega_1,\\
(1-|x|^2)/(2|x|), &\text{if}\ x\in\Omega_2,\\
|x| / 2, &\text{if}\ x\in\Omega_3,
\end{cases}
\]
and
\begin{equation}
\label{f:om3}
w(r\cos\theta, r\sin\theta)
:= 
\begin{cases}
c(\theta)/{r} & \text{on}\ \Omega_3,\\
0 & \text{on}\ \Omega\setminus\Omega_3, 
\end{cases}
\end{equation}
with 
\[
c\in L^1_+(-\pi/2, -\pi/4),\qquad
\int_{-\pi/2}^{-\pi/4} c(\theta)\, d\theta = \frac{\pi}{8}\,,
\]
is a solution to \eqref{f:MK} in the sense of
Definition~\ref{d:sol}.
Notice that, as usual, the function $\vf$ is obtained
solving the ODE \eqref{f:ode},
while $w$ is a solution to
$-\dive (w D\uo) = 0$ in $\Omega$, with
$\int_{\Omega_3} w = \Leb^2(\Omega_2)$.
Namely, given $\psi\in C^{\infty}_c(\R^2 \setminus \gafc)$ and denoting by
\[
I_j := \int_{\Omega_j}\left[(\vf+w)\pscal{D\uo}{D\psi} - \psi\right] dx,
\qquad j=1,2,3,
\]
passing in polar coordinates we easily get
\[
I_1 = 0,\qquad
I_2 = -\frac{\pi}{8}\, \psi(0), \qquad
I_3 = \psi(0) \int_{-\pi/2}^{-\pi/4} c(\theta)\, d\theta = \frac{\pi}{8}\, \psi(0),
\]
hence $I_1+I_2+I_3 = 0$ and the claim is proved.

%%%%%%%%%%%%%%%%%%%%
\subsection*{Details on Example~\ref{r:es3}}

The boundary datum corresponding to the geometry of
transport rays depicted in Figure~\ref{fig:nonuniq3}
is the continuous function
$\phi\colon\partial \Omega \to \R$
satisfying
\[
\phi(x)=
\begin{cases}
0 & x\in S_1:=\left((-1,0]\times \{-1,1\}\right )\cup \{(\cos \theta, \sin\theta): -\frac{\pi}{2}<\theta<-\frac{\pi}{4}\}  \\
1 & x\in S_2:=\left((0,1)\times\{0\}\right)\cup \left(\{1\}\times(0,1)\right) \\
1-\sqrt{x_1^2+x_2^2} & x\in S_3:=\{(x_1, -x_1), \ x_1\in (0, \frac{1}{\sqrt{2}})\}\\
x_1 &  x\in S_4:=(0,1)\times\{1\} \\
1-|x_2| & x\in S_5:=\{-1\}\times [-1,1]\,.
\end{cases}
\] 
%We consider the isotropic case with constant source $f=1$. 
The Lax--Hopf
function can be computed as follows:
\[
\uo(x)=
\begin{cases}
1-|x_2| & x \in \Omega_1 \\
1-\sqrt{x_1^2+x_2^2} & x \in \Omega_3 \\
1+x_2 & x\in \Omega_2':=\{x_1\in (0,1),\ 0<x_2<\frac{x_1^2}{4}\} \\
2-x_1 & x \in \Omega_2'':=\{x_2\in (0,1),\ 1-\frac{(1-x_2^2)}{4}<x_1<1\}\\
\sqrt{x_1^2+(1-x_2)^2} 
& x\in\Omega_2\setminus(\Omega_2'\cup\Omega_2'')\,.
\end{cases}
\]
Hence $\gaph=\gaf=S_1\cup S_2$.
Since (H1)--(H5) are satisfied,
by Theorem~\ref{t:biane} the pair
$(\uo, \vf)$ is a solution to \eqref{f:MK}.

On the other hand $(0,0)\in \gafc $ is the endpoint of all the transport rays covering $\Omega_3$, so that (H6) is not satisfied.
Hence, if $w\colon\Omega\to\R$ is a function 
of the form \eqref{f:om3},
%$w(r\cos\theta, r\sin\theta) := c(\theta) / r$ on $\Omega_3$, 
with
$c\in L^1_+(-\pi/2,\, -\pi/4)$, 
%and $w=0$ on $\Omega\setminus\Omega_3$,
then for every test function $\psi\in\test$ one has
\[
\int_{\Omega}w\pscal{D\uo}{D\psi} dx
= \int_{\Omega_3}w\pscal{D\uo}{D\psi} dx = 0,
\]
so that the pair $(\uo, \vf + w)$ is a solution
to \eqref{f:MK}.

%%%%%%%%%%%%%%%%%%%%%%%%%%%%%%%%%%%%%%%%%%%%%%%%%%%%%%%%%%%%%%%%%%%%%%%%%%%%%%%%%%%%
%\bibliographystyle{/Ga/BibTeX/mybst}
%\bibliography{/Ga/BibTeX/Graziano,/Ga/BibTeX/Ricerca,/Ga/BibTeX/Physics}
%\end{document}

\def\cprime{$'$}
\providecommand{\bysame}{\leavevmode\hbox to3em{\hrulefill}\thinspace}
\providecommand{\MR}{\relax\ifhmode\unskip\space\fi MR }
\providecommand{\MRhref}[2]{%
  \href{http://www.ams.org/mathscinet-getitem?mr=#1}{#2}
}
\providecommand{\href}[2]{#2}

\end{document}